\documentclass[10pt,twoside,a4paper,final]{amsart}
\usepackage[latin1]{inputenc}

\usepackage[centertags]{amsmath}%center-tags centered tags at split

\usepackage{amsthm}
\usepackage{amssymb}
\usepackage{amsfonts,a4}
\usepackage{amsxtra}
\usepackage{amscd}
\usepackage{mathabx} % some math-symbols as \bigstar
\usepackage{esint}

\usepackage[english]{babel}
\usepackage{algorithm}
\usepackage{algorithmic}

\usepackage{color}
\usepackage{todonotes}
\usepackage[mathscr]{eucal}

\usepackage{bbm}
\usepackage{enumerate}

\usepackage{graphicx}
\graphicspath{{./plots/}}

\usepackage{xspace} %
\usepackage[notcite,notref]{showkeys}%shows labels at sidebar 

 \usepackage[scrtime]{prelim2e}

\usepackage{tikz}
\usetikzlibrary{calc,arrows}

\newtheorem{thm}{Theorem}
\newtheorem{cor}[thm]{Corollary}
\newtheorem{lem}[thm]{Lemma}
\newtheorem{prop}[thm]{Proposition}
\newtheorem{rem}[thm]{Remark}
\theoremstyle{definition}

\newtheorem{algo}[thm]{Algorithm}

\numberwithin{equation}{section}

\newcommand{\ADGM}[1][]{\textsf{ADGM}#1\xspace}
\newcommand{\AFEM}[1][]{\textsf{AFEM}#1\xspace}
 
\newcommand{\Ao}{\ensuremath{\mathcal{A}}}
\newcommand{\Aon}[1][n]{\ensuremath{\Ao_{n}}}
\newcommand{\abs}[1]{\ensuremath{\left|#1\right|}}

\newcommand{\bi}[1][\grid]{\mathfrak{B}_{#1}}
\newcommand{\bilin}[3][\grid]{\ensuremath{\bi[#1][#2,\,#3]}}

\newcommand{\Cleq}{\ensuremath{\lesssim}}

\newcommand{\definedas}{\mathrel{:=}}
\newcommand{\dual}[2]{\ensuremath{\left\langle #1,\,#2\right\rangle}}

\newcommand{\dx}{\ensuremath{\,\mathrm{d}x\xspace}}
\newcommand{\ds}{\ensuremath{\,\mathrm{d}s\xspace}}

\DeclareMathOperator{\divo}{div}
\newcommand{\DGFEM}[1][]{\textsf{DGFEM}#1\xspace}

\newcommand{\elm}{\ensuremath{E}\xspace}
\newcommand{\enorm}[2][\grid]{\left|\negthinspace\left|\negthinspace\left|{#2}%
                      \right|\negthinspace\right|\negthinspace\right|_{#1}}

\newcommand{\est}{\mathcal{E}}
\newcommand{\ESTIMATE}{\textsf{ESTIMATE}}

\newcommand{\grids}{\ensuremath{\mathbb{G}}\xspace}
\newcommand{\grid}{\mathcal{G}}

\newcommand{\gridk}[1][k]{\grid_{#1}}

\newcommand{\helm}[1][\ell]{\ensuremath{h_{\elm}}}
\newcommand{\hk}[1][k]{\ensuremath{h_{#1}}}

\newcommand{\hG}[1][\grid]{\ensuremath{h_{#1}}}

\newcommand{\jump}[1]{\left[\negthinspace\left[{#1}\right]\negthinspace\right]}

\DeclareMathAlphabet{\lf}{OT1}{pzc}{m}{it}

\newcommand{\liftG}[1][\grid]{\ensuremath{L_{#1}}}
\newcommand{\liftS}[1][\grid]{\ensuremath{L_{#1}^\side}}
\newcommand{\LDG}{\textsf{LDG}\xspace}

\newcommand{\mean}[1]{\{ \kern -1.6mm \{#1\} \kern -1.6mm \}}
\newcommand{\marked}{\mathcal{M}}
\newcommand{\markedk}[1][k]{\marked_{#1}}

\newcommand{\MARK}{\textsf{MARK}\xspace}

\newcommand{\nablaG}[1][\grid]{\ensuremath{\nabla_{\!\texttt{pw}}}}
\newcommand{\nablak}[1][k]{\ensuremath{\nabla_{k}}}
\newcommand{\N}{\ensuremath{\mathbb{N}}}

\newcommand{\neigh}{\ensuremath{{N}}}
\newcommand{\neighk}[1][k]{\ensuremath{{N}_{#1}}}
\newcommand{\neighG}[1][\grid]{\ensuremath{{N}_{#1}}}

\newcommand{\nodes}{\ensuremath{\mathcal{N}}}
\newcommand{\normal}{\ensuremath{\vec{n}}}

\newcommand{\norm}[2][\Omega]{\ensuremath{\left\|#2\right\|_{#1}}}

\newcommand{\NIPG}{\textsf{NIPG}\xspace}

\DeclareMathOperator{\osc}{osc}

\newcommand{\omegaG}[1][\grid]{\ensuremath{\omega_{#1}}}
\newcommand{\omegak}[1][k]{\ensuremath{\omega_{#1}}}
\renewcommand{\P}{\ensuremath{\mathbb{P}}}

\renewcommand{\paragraph}[1]{\noindent\raisebox{0pt}[10pt][0pt]{\textbf{#1.}}}

\newcommand{\REFINE}{\textsf{REFINE}\xspace}
\newcommand{\R}{\ensuremath{\mathbb{R}}}

\newcommand{\riftG}[1][\grid]{\ensuremath{R_{#1}}}
\newcommand{\riftS}[1][\grid]{\ensuremath{R_{#1}^\side}}
\newcommand{\scp}[3][\Omega]{\ensuremath{\left\langle #2,\,#3\right\rangle}_{#1}}

\newcommand{\side}{\ensuremath{S}\xspace}
\newcommand{\sides}{\mathcal{S}}

\newcommand{\SIPG}{\textsf{SIPG}\xspace}

\newcommand{\SOLVE}{\textsf{SOLVE}}
\DeclareMathOperator{\supp}{supp}

\newcommand{\ud}{\mathrm{d}}
\newcommand{\uG}[1][\grid]{\ensuremath{u_{#1}}}
\newcommand{\uk}[1][k]{\ensuremath{u_{#1}}}

\renewcommand{\vec}[1]{\ensuremath{\boldsymbol{#1}}}

\newcommand{\V}{\ensuremath{\mathbb{V}}}

\newcommand{\Vk}[1][k]{\V_{#1}}
\newcommand{\VG}[1][\grid]{\V(#1)}

\begin{document}
%\svnInfo $Id: paris.tex 764 2010-08-30 18:13:20Z kreuzer $ 

% own counter-----------------------------------------------------
%\newcounter{reserve}{1}

% head -----------------------------------------------------------

\title[Convergence of \ADGM]{Convergence of adaptive \\discontinuous
  Galerkin methods\\{\mdseries\scriptsize (corrected version of [Math. Comp. 87 (2018), no.~314,~2611--2640])}}

\author[Ch.~Kreuzer]{Christian Kreuzer}
\address{Christian Kreuzer,
 Fakult\"at f\"ur Mathematik,
 Ruhr-Universit\"at Bochum,
 Universit\"atsstrasse 150, D-44801 Bochum, Germany
 }%
\urladdr{http://www.ruhr-uni-bochum.de/ffm/Lehrstuehle/Kreuzer/index.html}
\email{christan.kreuzer@rub.de}

\thanks{The research of Christian Kreuzer was supported by the DFG research grant KR 3984/5-1.}

\author[E.\,H.~Georgoulis]{Emmanuil H.~Georgoulis}
\address{Emmanuil H.~Georgoulis, Department of Mathematics,
University of Leicester,
University Road, 
Leicester, LE1 7RH,
United Kingdom and Department of Mathematics, School of Applied Mathematical and Physical Sciences, National Technical University of Athens, Zografou 157 80, Greece 
 }%
\urladdr{http://www.le.ac.uk/people/eg64}
\email{Emmanuil.Georgoulis@le.ac.uk}

\thanks{Emmanuil H.~Georgoulis acknowledges support by the Leverhulme Trust.}



\begin{abstract} 
  We develop a general convergence theory for adaptive discontinuous
  Galerkin methods for elliptic PDEs covering the popular SIPG, NIPG
  and LDG schemes as
  well as all practically relevant marking strategies. Another key feature of the
  presented result is, that it holds for penalty parameters only necessary for
  the standard analysis of the respective scheme. The analysis is
  based on a quasi interpolation into a newly developed limit space
  of the adaptively created non-conforming discrete spaces, 
  which enables to generalise the basic convergence result for
  conforming adaptive finite element methods by Morin, Siebert, and
  Veeser [A basic convergence result for conforming adaptive finite
  elements, Math. Models Methods Appl. Sci., 2008, 18(5),
  707--737].
\end{abstract}

\date{\small\today}

\maketitle

\section{Introduction\label{s:introduction}}

% Computable a posteriori error estimates are an important part of
% computational practice because of their ability to provide computable
% information about errors and drive adaptive mesh refinement
% algorithms. The theory and performance of various classes of a
% posteriori bounds for elliptic partial differential equation (PDE)
% problems is well understood by now. A posteriori estimates are
% typically used in the context of adaptive algorithms, which are
% iterative in nature. The convergence theory of adaptive finite element
% algorithms has seen unprecedented development during the last 15
% years. Convergence results for various finite element methods exist
% for a number of adaptive strategies \cite{}, predominantly for
% elliptic problems \cite{}, while few results exist for parabolic
% problems \cite{}. In the context of adaptive algorithms for
% conforming finite element methods, instance optimality was also recently proven, showing
% that .... 

 Discontinuous Galerkin finite element methods (\DGFEM) have 
 enjoyed considerable attention during the last two decades,
 especially in the context of adaptive algorithms (\ADGM[s]): the absence of any
 conformity requirements across element interfaces characterizing
 \DGFEM approximations allows for extremely general adaptive meshes
 and/or an easy implementation of variable local polynomial degrees in the
 finite element spaces.  There has been a substantial activity in
 recent years for 
the derivation of a posteriori bounds for discontinuous Galerkin
methods for elliptic problems 
\cite{KarakashianPascal:03,BeckerHansboLarson:2003,Ainsworth:07,HoustonSchoetzauWihler:07,CarstensenGudiJensen:09,ErnVohralik:09,ErnStephansenVohralik:10,ZhuGianiHouston:11,DiPietroErn:12}. 
Such a posteriori estimates are an essential building block in the
  context of adaptive algorithms, which typically consist of a loop 
\label{eq:SEMR}
\begin{equation}\label{SEMR}
  \textsf{SOLVE $\rightarrow$ ESTIMATE $\rightarrow$
    MARK  $\rightarrow$ REFINE}.
\end{equation}
The convergence theory, however, for the `extreme' non-conformity
case of \ADGM[s] had been a particularly challenging problem due to the
presence of a negative power of the mesh-size $h$ stemming from the
discontinuity-penalization term. As a consequence, the error is not necessarily
monotone under refinement. 
Indeed, consulting the unprecedented developments of convergence and
optimality theory of conforming adaptive finite element methods (\AFEM[s]) during
the last two decades, the strict reduction of some error quantity appears
to be fundamental for most of the results. 
In fact, D\"orfler's marking 
strategy typically ensures that the error is uniformly reduced in
each iteration~\cite{Doerfler:96,MoNoSi:00,MoNoSi:02}
and leads to optimal convergence
rates~\cite{Stevenson:07,CaKrNoSi:08,KreuzerSiebert:11,DieningKreuzer:08,BelenkiDieningKreuzer:12}; 
compare also with the 
monographs \cite{NoSiVe:09,CarstensenFeischlPraetorius:2014} and the 
references therein. Showing that 
the error reduction is proportional to the estimator on the refined
elements, instance
optimality of an adaptive finite element method was shown recently for
an \AFEM with modified marking strategy in
\cite{DieningKreuzerStevenson:2016,KreuzerSchedensack:2016}. 
A different approach was, however,  taken in \cite{MorinSiebertVeeser:08,Siebert:11}, where 
convergence of the \AFEM is proved, exploiting that the approximations
converge to a solution in the closure of the
adaptively created finite element spaces in the trial
space together with standard properties of the a posteriori bounds. 
The result covers a large class of inf-sup stable PDEs and all
practically relevant marking strategies without yielding convergence rates though.

Karakashian and Pascal \cite{KarakashianPascal:07} gave the first
proof of convergence for an adaptive \DGFEM based on a symmetric
interior penalty scheme (\SIPG) with D\"orfler marking for Poisson's problem. Their
proof % is tailored to the symmetry of \SIPG and [MANOLIS]
addresses the
challenge of negative power of $h$ in the estimator, by showing that
the discontinuity-penalization term can be controlled by the element
and jump residuals only, provided that the \DGFEM
discontinuity-penalisation parameter, henceforth denoted by $\sigma$,
is chosen to be \emph{sufficiently large}; the element and jump
residuals involve only positive powers of $h$ and, therefore, 
can be controlled similarly as for conforming methods. The optimality
of the adaptive \SIPG was shown in \cite{BonitoNochetto:10}; see
also \cite{HoppeKanschatWarburton:09}. 

The standard error analysis of the \SIPG requires
that $\sigma$ is sufficiently large for the respective bilinear from
to be coercive with respect to an energy-like norm. It is not known in
general, however, whether the choice of $\sigma$ required for
coercivity of the interior penalty \DGFEM bilinear form is large enough
to ensure that the discontinuity-penalization term can be controlled
by the element and jump residuals only. Therefore, the convergence of
\SIPG is still open for values of $\sigma$ large
enough for coercivity but, perhaps, \emph{not large enough} for the
crucial result from \cite{KarakashianPascal:07} to hold. To the best
of our knowledge, the only result in this direction is the proof of
convergence of a weakly overpenalized \ADGM for linear elements
\cite{GudiGuzman:14}, utilizing the intimate relation between this
method and the lowest order Crouzeix-Raviart elements.

This work is concerned with proving that the \ADGM converges for
\emph{all} values of $\sigma$ for which the method is coercive,
thereby settling the above discrepancy between the magnitude of
$\sigma$ required for coercivity and the, typically much larger,
values required for proof of convergence of \ADGM. Apart from
settling this open problem theoretically, this new result has some
important consequences in practical computations: it is well known
that as $\sigma$ grows, the condition number of the respective
stiffness matrix also grows. Therefore, the magnitude of the
discontinuity-penalization parameter $\sigma$ affects the performance
of iterative linear solvers, whose complexity is also typically
included in algorithmic optimality discussions of adaptive finite
elements. In addition, the theory presented here includes a large
class of practically
relevant marking strategies and covers popular 
discontinuous Galerkin methods like the local discontinuous Galerkin
method (LDG) and even the \emph{nonsymmetric} interior penalty method
(\NIPG), which are coercive for any $\sigma>0$. Moreover, we
  expect that it can be generalised to non-conforming
  discretisations for a number of other problems
  like the Stokes equations or fourth order elliptic problems.
  However, as for the
  conforming counterpart~\cite{MorinSiebertVeeser:08}, no
  convergence rates are guaranteed.

The proof of convergence of the \ADGM, discussed
below, is motivated by the basic convergence for the conforming adaptive
finite element framework of Morin, Siebert and Veeser
\cite{MorinSiebertVeeser:08}. More specifically, we extend
considerably the ideas from  \cite{MorinSiebertVeeser:08} and 
\cite{Gudi:10} to be able
to address the crucial challenge that the limits of \DGFEM solutions,
constructed by the adaptive algorithm, do not necessarily belong to
the energy space of the boundary value problem as well as to conclude
convergence from a perturbed best approximation result. 

To highlight the key
theoretical developments without the need to resort to complicated
notation, we prefer to focus on the simple setting of the Poisson
problem with essential homogeneous boundary conditions and conforming
shape regular triangulations. We believe,
however, that the results presented below are valid for general
elliptic PDEs including convection and reaction phenomena as well as
for some classes of non conforming meshes; compare with \cite{BonitoNochetto:10}.  
% \todo{Remark: Alexander checked that the result can be generalised
%   to finite levels of hanging nodes}

The remainder of this work is structured as follows. In
Section~\ref{sec:prelim} we shall introduce the \ADGM 
framework for Poisson's equation and state the main result, which is
then proved in Section~\ref{sec:convergence-est} after 
some auxiliary results,
needed to generalise \cite{MorinSiebertVeeser:08},
are provided in Sections~\ref{sec:limit} and~\ref{sec:uinfty}. 
In particular, in Section~\ref{sec:limit} a space is presented,
which is generated from limits of discrete discontinuous functions in
the sequence of discontinuous Galerkin spaces constructed by
\ADGM. Section~\ref{sec:uinfty} is then concerned with proving that  
the sequence of discontinuous Galerkin solutions produced by \ADGM
converges indeed to a generalised Galerkin solution in this limit space.
This follows from an (almost) best-approximation property,
generalising the ideas in \cite{Gudi:10}.

\section{The \ADGM and the main result}
\label{sec:prelim}
Let a measurable set
$\omega$ and a $m\in\N$. We consider the Lebesgue
space $L^2(\omega;\R^m)$ of square integrable
functions over $\omega$ with values in $\R^m$, with inner product  $\scp[\omega]{\cdot}{\cdot}$ 
and associated norm $\norm[\omega]{\cdot}$. We also set $L^2(\omega):=L^2(\omega;\R)$.
The
Sobolev space $H^1(\omega)$ is the space of all functions in
$L^2(\omega)$ whose weak gradient is in $L^2(\omega;\R^d)$, for $d\in \N$. %  For
% $D\subset\partial\omega$, with non vanishing $(d-1)$-dimensional Hausdorff
% measure, let 
% $H_D^1(\omega)$ be the space of functions in $H^1(\omega)$ with
% vanishing trace on $D$ and denote  $H_0^1(\omega):=H_{\partial\omega}^1(\omega)$.
Thanks to the Poincar\'e-Friedrichs' inequality, the closure
$H_0^1(\omega)$ of $C_0^\infty(\omega)$ in $H^1(\omega)$ 
is a Hilbert space with inner product
$\scp[\omega]{\nabla \cdot}{\nabla \cdot}$ and norm
$\norm[\omega]{\nabla \cdot}$.  
Also, we denote
the dual space $H^{-1}(\omega)$ of $H_0^1(\omega)$, with the norm
$\norm[H^{-1}(\omega)]{\lf{v}}:=\sup_{w\in
  H_0^1(\omega)}\frac{\dual{\lf{v}}{w}}{\norm[\omega]{\nabla w}}$, $\lf{v}\in H^{-1}(\omega)$, with
dual brackets defined by $\dual{\lf{v}}{w}:=\lf{v}(w)$, for
$w\in H^1_0(\omega)$.

Let $\Omega \subset \mathbb{R}^d$,
$d=2,3$, be a bounded polygonal ($d=2$) or polyhedral ($d=3$)
Lipschitz domain. We consider the Poisson problem
\begin{equation}\label{eq:elliptic}
  -\Delta u=f\quad\text{in}\quad\Omega,
  \qquad
  u=0\quad\text{on}\quad\partial\Omega,
\end{equation}
with $f\in L^2(\Omega)$. The weak formulation of \eqref{eq:elliptic} reads:
find $u\in H^1_0(\Omega)$, such that 
\begin{align}\label{eq:weak}
  \scp[\Omega]{\nabla u}{\nabla v}=\scp[\Omega]{f}{v}\qquad \text{for
    all}~v\in H_0^1(\Omega).
\end{align}
From the Riesz representation theorem, it follows that the solution
$u$ exists and is unique.

\subsection{Discontinuous Galerkin method}
  \label{secdgfem}

Let $\grid$ be a conforming (that is, not containing any hanging nodes) 
subdivision of $\Omega$ into disjoint
closed simplicial % ,  quadrilateral ($d=2$),  or hexahedral ($d=3$)
elements $\elm$  so that $\bar{\Omega}=\bigcup\{\elm:\elm\in\grid\}$
and set $h_{\elm}:=|{\elm}|^{1/d}$.  
% By shape-regularity, we mean more precisely 
% that 
% \begin{align*}
%   \sup_{\elm\in\grid}\diam(\elm)/h_\elm<\infty.
% \end{align*}
Let $\sides=\sides(\grid)$ be the set of $(d-1)$-dimensional element faces
$\side$ associated with the subdivision $\grid$ including $\partial
\Omega$, and let
$\mathring\sides=\mathring\sides(\grid)\subset\sides$ by the subset of 
interior faces only. 
We also introduce the
\emph{mesh size} function $\hG:\Omega\to \mathbb{R}$, defined by $\hG(x):=
h_\elm$, if $x\in \elm\backslash\partial\elm$ and $\hG(x)=h_\side:=
|S|^{1/(d-1)} $, if $x\in S\in\sides$
and set
  $\Gamma=\Gamma(\grid)=\bigcup\{S:S\in\sides\}$ and $\mathring\Gamma=\mathring\Gamma(\grid)=\bigcup\{S:S\in\mathring\sides\}$.
We assume that $\grid$ is derived by iterative or recursive newest vertex bisection of
an initial conforming mesh $\grid_0$; see \cite{Baensch:91,Kossaczky:94,Maubach:95,Traxler:97}.
We denote  by $\grids$ the family of shape regular triangulations consisting of such
subdivisions of $\grid_0$.

Let
$\mathcal{P}_r({\elm})$ denote the the space of all polynomials on ${\elm}$
of  degree at most $r\in\N$, we define the discontinuous finite element space
\begin{equation}\label{eq:FEM-spc}
  \VG  :=\prod_{\elm\in\grid}\P_r(\elm)\subset \prod_{\elm\in\grid}W^{1,p}(\elm)=:
  W^{1,p}(\grid), \quad 1\le p\le\infty,
\end{equation}
and $H^1(\grid):=W^{1,2}(\grid)$.
Let $\nodes=\nodes(\grid)$ be the set of Lagrange nodes of $\VG$ and
define the neighbourhood of a node $z\in\nodes(\grid)$ by
$\neighG(z):=\{\elm'\in\grid:z\in\elm'\}$, 
and the union of its elements by
$\omegaG(z)=\bigcup\{\elm'\in\grid:z\in\elm'\}$. We also define the corresponding neighbourhoods for all elements $\elm\in \grid$ by
$\neighG(\elm):=\{\elm'\in\grid:\elm\cap\elm'\neq\emptyset\}$
and
$\omegaG(\elm)=\bigcup\{\elm'\in\grid:\elm'\cap\elm\neq\emptyset\}=\bigcup\{\omegaG(z):z\in
\nodes(\elm)\cap\elm\}$, respectively, and set  $\omegaG(S):=\bigcup\{\elm\in\grid:S\subset\elm\}$; compare with Figure~\ref{fig:neigh}.
 The numbers of neighbours
$\#\neighG(z)$ and $\#\neighG(\elm)$ are uniformly bounded for all
$z\in\nodes$, respectively
$\elm\in\grid$, depending on the shape regularity of $\grid$ and, thus,
on $\grid_0$.

\begin{figure}[h]
  \label{fig:prichn2b}
  \centering
  \begin{tikzpicture}[scale=2.5]
    \draw [thick]
    (-0.2,0.2) coordinate (A) -- 
    (1,0) coordinate (B) -- 
    (0,1) coordinate (C) -- cycle;
    \draw [thick]
    (B) -- 
    (1.5,0.5) coordinate (D) --
    (C);
    \draw [thick]
    (D) -- 
    (0.6,1.3) coordinate (E) --
    (C);
    \draw [thick]
    (C) -- 
    (E) --
    (-0.5,1.25) coordinate (F);
    \draw [thick]
    (C) -- 
    (F) --
    (A);
    \draw [thick]
    (A) -- 
    (F) --
    (-0.8,-0.4) coordinate (G);
    \draw [thick]
    (A) -- 
    (G) --
    (B);
    \draw [thick]
    (G) -- 
    (1.2,-0.5) coordinate (H)--
    (B);
    \draw [thick]
    (H) -- 
    (D);
    \fill (0.25,0.4) coordinate (bot) node
    {$\elm$};
  \end{tikzpicture}
  \caption{The neighbourhood $\neighG(\elm)$ of some $\elm\in\grid$.\label{fig:neigh}}
\end{figure}

Let $\elm^+$, $\elm^-$ be two generic elements sharing a face
$S:=\elm^+\cap\elm^-\in\mathring\sides$ and let $\normal^+$ and
$\normal^-$ the outward normal vectors of $\elm^+$ respectively
$\elm^-$ on $S$. For
$q:\Omega\to\mathbb{R}$ and $\vec{\phi}:\Omega\to\mathbb{R}^d$, let
$q^{\pm}:=q|_{S\cap\partial\elm^{\pm}}$ and
$\vec{\phi}^{\pm}:=\vec{\phi}|_{S\cap\partial\elm^{\pm}}$, and set
\begin{alignat*}{2}
\mean{q}|_\side&:=\frac12(q^+ + q^-),\ &\qquad
\mean{\vec{\phi}}|_\side &:=\frac12(\vec{\phi}^+ + \vec{\phi}^-),
\\
\jump{q}|_\side &:=q^+\normal^++q^-\normal^-,\ &\qquad
\jump{\vec{\phi}}|_\side &:=\vec{\phi}^+\cdot \normal^++\vec{\phi}^-\cdot
\normal^-;
\end{alignat*}
if $S\subset \partial\elm\cap\partial\Omega$, we set
$\mean{\vec{\phi}}|_\side :=\vec{\phi}^+$  and
$\jump{q}|_\side :=q^+\normal^+$.

In order to define the discontinuous Galerkin schemes, we
  introduce the following local lifting operators. For $S\in\sides$,
  we define  $\riftS:L^2(\side)^d\to
  \prod_{\elm\in\grid}\P_\ell(\elm)^d$ and $\liftS:L^2(\side)\to
  \prod_{\elm\in\grid}\P_\ell(\elm)^d$ by 
  \begin{subequations}\label{eq:liftG}
    \begin{align}
      \int_\Omega \riftS(\vec{\phi})\cdot
      \vec{\tau}\dx&=\int_\side \vec{\phi}\cdot\mean{\vec{\tau}}\ds\qquad\forall
                     \vec{\tau}\in \prod_{\elm\in\grid}\P_\ell(\elm)^d
                     \intertext{and}
                     \int_\Omega \liftS(q)\cdot
                     \vec{\tau}\dx&=\int_\side q\jump{\vec{\tau}}\ds\qquad\forall
                                    \vec{\tau}\in \prod_{\elm\in\grid}\P_\ell(\elm)^d,
    \end{align}
    with $\ell\in\{r,r+1\}$.
    Note that $\liftS(q)$ and $\riftS(\vec{\phi})$ vanish outside 
    $\omegaG(S)$. Moreover, using the local definition and the
    boundedness of the lifting operators in a reference situation
    together with standard scaling arguments, we have 
    for $\vec{\phi}\in\P_r(\side)^d$ and
    $q\in\P_r(\side)$ that
    \begin{align}\label{eq:liftGstab}
      \norm{\liftS(\vec{\phi})}\Cleq
      \norm[\side]{\hG^{-1/2}\vec{\phi}}\quad\text{and}\quad 
       \norm{\riftS(q)}\Cleq
      \norm[\side]{\hG^{-1/2}q};
    \end{align}
    \end{subequations}
    compare with \cite{ArnBreCocMar:02}. 
    Also, here and below we write $a \lesssim b$
    when $a \le Cb$ for a constant $C$ not depending on the local mesh size of $\grid$ or other essential quantities
    for the arguments presented below. 
    Observing that the sets $\omegaG(S)$, $\side\in\sides$ do
    overlap  at most $d+1$ times, we have for the 
    global lifting operators $\riftG:L^2(\Gamma)^d\to \VG^d$ and
    $\liftG:L^2(\mathring\Gamma)\to \VG^d$ defined by 
    \begin{align*}
      \riftG(\vec{\phi}):=\sum_{\side\in\sides}\riftS(\vec{\phi})\qquad\text{and}\qquad \liftG(q):=\sum_{\side\in\mathring\sides}\riftS(q),
    \end{align*}
  that 
    \begin{align*}
      \norm{\riftG(\jump{v})}\Cleq 
      \norm[\Gamma]{\hG^{-1/2} v}
      \quad\text{and}\quad 
      \norm{\liftG(\vec{\beta}\cdot\jump{v})}\Cleq \abs{\vec{\beta}}
      \norm[\mathring\Gamma]{\hG^{-1/2} v}
    \end{align*}
    for all $v\in \VG$ and $\vec{\beta}\in \R^d$.

We define the bilinear form $\bilin{\cdot}{\cdot}:\VG\times \VG\to
\mathbb{R}$ by
\begin{align}\label{dgbilinear}
\begin{aligned}
 \bilin{w}{v}&:=\int_{\grid}\nabla w\cdot\nabla v\,\ud x
       -\int_{\sides}\big(\mean{\nabla w}\cdot\jump{v}+\theta \mean{\nabla v}\cdot\jump{w}\big)\ds
       \\
       &\quad +\int_{\mathring\sides}
       \big(\vec{\beta}\cdot\jump{w}\jump{\nabla v}+\jump{\nabla
         w}\vec{\beta}\cdot\jump{v}\big)\ds
       \\
       &\quad + \int_\Omega \gamma \big(\riftG(\jump{w})+\liftG(\vec{\beta}\cdot\jump{w})\big)\cdot \big(\riftG(\jump{v})+\liftG(\vec{\beta}\cdot\jump{v})\big)\dx
       \\
       &\quad+\int_{\sides}\frac{\sigma}{\hG}\jump{w}\cdot\jump{v}\,\ud s;
\end{aligned}
\end{align}
for $\theta\in\{\pm1\}$, $\gamma\in\{0,1\}$, $\vec{\beta}\in\R^d$ and $\sigma\ge0$.
Here we have used the short-hand notation
\begin{align*}
  \int_\grid \cdot\dx:=\sum_{\elm\in\grid}\int_\elm\cdot\dx\qquad\text{and}\qquad\int_\sides \cdot\ds:=\sum_{S\in\sides}\int_\side  \cdot\ds.
\end{align*}
We consider the choices $\theta=1$, $\vec{\beta}=\vec{0}$, and $\gamma=0$ yielding the
symmetric interior penalty method (\SIPG) \cite{DouglasDupont:1976}, $\theta=-1$,
$\vec{\beta}=\vec{0}$, and $\gamma=0$ which gives the nonsymmetric interior
penalty methods (\NIPG) \cite{RiviereWheelerGirault:1999}, and $\theta=1$, $\vec{\beta}\in\R^d$, and
$\gamma=1$ which yields the local discontinuous Galerkin method (\LDG) \cite{CockburnShu:1998};
compare also with~\cite{ArnBreCocMar:02} and~\cite{John2016}.

In all three cases, the corresponding \emph{discontinuous
Galerkin finite element method} (\DGFEM) then reads: find $u_\grid\in \VG$ such that
\begin{equation}\label{ipdg}
\bilin{u_\grid}{v_\grid}=\int_{\Omega} f v_\grid\,\ud x =:l(v_\grid),\quad\text{for all }\ v_\grid\in \VG.
\end{equation}
Upon denoting by $\nablaG v$ the piecewise
gradient $\nablaG v|_\elm=\nabla v|_\elm$ for all $\elm\in
\grid$, the corresponding energy norm $\enorm{\cdot}$ is defined by
\begin{align*}
 \enorm{w}&:=\Big(\norm{\nablaG w}^2
    +\bar\sigma\norm[\Gamma]{\hG^{-1/2}\jump{w}}^2\Big)^{1/2},
%    \\
%    &:=\norm{w}^2+\int_\Omega\abs{\nablaG w}^2\dx
%    +\int_\sides\abs{\hG^{-1/2}\jump{w}}^2\ds
\end{align*}
for $w|_{\elm}\in H^1(\elm)$, $\elm\in\grid$. Here $\bar\sigma:=\max\{1,\sigma\}$.
Also, for some
subset $\mathcal{M}\subset\grid$ with $\omega=\bigcup\{\elm\mid \elm\in\mathcal{M}\}$, we define 
\begin{align*}
  \enorm[\mathcal{M}]{w}&:=\Big(\norm[\omega]{\nablaG w}^2
    +\bar\sigma\norm[\Gamma(\mathcal{M})]{\hG^{-1/2}\jump{w}}^2\Big)^{1/2}.
\end{align*}

If for \SIPG we have $\sigma:=C_{\sigma} r^2$ for some
constant $C_{\sigma}>0$ sufficiently large, $\sigma>0$ for \NIPG and
for \LDG~$\sigma>0$ when $\ell=r$ and $\sigma=0$ when $\ell=r+1$
  (\cite{John2016}), 
then there exists $\alpha=\alpha(\sigma)>0$, such that 
\begin{align}\label{eq:coercive}
  \alpha\enorm{w}^2\le \bilin{w}{w}\qquad\forall w\in H^1(\grid),
\end{align}
i.e. all three \DGFEM[s] are
coercive in $\VG$. Note that coercivity~\eqref{eq:coercive} holds true also for functions in $H^1(\grid)$ after extending the discrete bilinear form using some liftings; see, e.g.,
\cite{Arnold:82,ArnBreCocMar:02,John2016} for details. 
The choice $\bar\sigma=\max\{1,\sigma\}$ accounts
for the fact that we can have $\sigma=0$ for the LDG
in~\cite{John2016}.

From standard scaling arguments, we conclude
 the following local Poincar\'e-Friedrichs inequality
 from~\cite{Brenner:2003,BuffaOrtner:09}. 

\begin{prop}[Poincar\'e-$\VG$]\label{prop:poincareG}
  Let $\grid$ be a triangulation of $\Omega$ and $\grid_\star$ some
  refinement of $\grid$. Then, for $v\in \V(\grid_\star)$,
  $\elm\in\grid$ and $v_\elm:=|\omegaG(\elm)|^{-1}\int_{\omegaG(\elm)}
  v\dx$, we have 
  \begin{align*}
    \norm[\omegaG(\elm)]{v-v_\elm}^2\Cleq \int_{\omegaG(\elm)}\hG^2|\nablaG
    v|^2\dx+\int_{S\in\sides_\star, S\subset\omegaG(\elm)}\hG^2\hG[\grid_\star]^{-1}\jump{v}^2\ds,
  \end{align*}
  where $\sides_\star=\sides(\grid_\star)$ and the hidden constant
  depends on $d$ and on the shape regularity of $\neighG(\elm)$.
\end{prop}

The next important result from \cite[Theorem 2.2]{KarakashianPascal:03} (compare also
with  \cite[Lemma 6.9]{BonitoNochetto:10} and \cite[Theorem 3.1]{BuffaOrtner:09}) quantifies the local distance
of a discrete non-conforming function to the conforming subspace with
the help of the of the scaled jump terms. 
\begin{prop}\label{P:dist-dGcG}
  For $\grid\in\grids$, there exists an interpolation operator
  $\mathcal{I}_\grid:H^1(\grid)\to\VG\cap H_0^1(\Omega)$, such that we have 
  \begin{align*}
    \norm[\elm]{\hG^{-1/2}(v-\mathcal{I}_\grid
    v)}^2+\norm[\elm]{\nabla(v-\mathcal{I}_\grid v)}^2
    \Cleq \int_{\partial\elm}\hG^{-1}\jump{v}^2\ds,
  \end{align*}
  for
  all $\elm\in\grid$ and $v\in\VG$.   % The hidden constant depends only
  % on the shape regularity of $\grid$ the polynomial degree $r$ and the
  % dimension $d$.
\end{prop} 
From this, we can easily deduce the following broken Friedrichs type
inequality; compare also with \cite[(4.5)]{BuffaOrtner:09}.
\begin{cor}[Friedrichs-$\VG$]\label{C:Friedrichs}
  Let $\grid\in\grids$, then
  \begin{align*}
    \norm[\Omega]{v} \Cleq \enorm[\grid]{v}\quad\text{for all}~v\in\VG.
  \end{align*}
  % The hidden constant depends only
  % on the shape regularity of $\grid$ the polynomial degree $r$ and the
  % dimension $d$.
\end{cor}

Let
$BV(\Omega)$ denote the Banach space of functions with bounded
variation equiped with the norm 
\begin{align*}
  \norm[BV(\Omega)]{v}=\norm[L^1(\Omega)]{v}+|Dv|(\Omega),
\end{align*}
where $Dv$ is the measure representing the distributional derivative
of $v$ with total variation 
\begin{align*}
  |Dv|(\Omega)=\sup_{\phi\in C_0^1(\Omega)^d, \norm[L^{\infty(\Omega)\le
  1}]{\phi}}\int_\Omega v\divo \phi\dx.
\end{align*}
Here the supremum is taken over the space $C_0^1(\Omega)^d$ of all vector
valued continuously differentiable functions with compact support in $\Omega$.

Another crucial result \cite[Lemma~2]{BuffaOrtner:09} states then that the 
total variation of the distributional derivative of broken Sobolev functions 
is bounded by the discontiuous Galerkin norm.
\begin{prop}\label{P:|Dv|<dG} 
  For $\grid\in\grids$ we have that
  \begin{align*}
    |Dv|(\Omega)\Cleq \norm[L^1(\Omega)]{
    v}+\int_{\sides}\abs{\jump{v}}\ds\Cleq \enorm[\grid]{v}\quad\text{for all}~v\in H^{1}(\grid).
  \end{align*}
  % The hidden constant depends only
  % on the shape regularity of $\grid$ and the
  % dimension $d$. 
\end{prop}

\subsection{A posteriori error bound}
\label{sec:aposteriori}

We recall the a posteriori results from
\cite{KarakashianPascal:03,BonitoNochetto:10,BustinzaGaticaCockburn:2005,BeckerHansboLarson:2003};
compare also with \cite{CarstensenGudiJensen:09}. 

For $v\in \V(\grid)$, we define the local error indicators for
$\elm\in\grid$ by 
\begin{align*}
  \est_\grid(v,\elm):=\Big(\int_{\elm}\hG^2|f+\Delta
      v|^2\dx+\int_{\partial\elm\cap\Omega} \hG\jump{\nabla
  v}^2\ds+\sigma\int_{\partial\elm}\hG^{-1}\jump{v}^2\ds\Big)^{1/2};
\end{align*}
when $v=\uG$, we shall write $\est_\grid(\elm):=\est_\grid(\uG,\elm)$.
Also, for $\mathcal{M}\subset\grid$, we set 
\begin{align*}
  \est_\grid(v,\mathcal{M}):=\Big(\sum_{\elm\in\mathcal{M}}\est(v,\elm)^2\Big)^{1/2}.
\end{align*}

\begin{prop}% [{\cite[Lemma 3.1]{BonitoNochetto:10}% \cite[Theorem~3.1]{KarakashianPascal:03}
  % }]
  \label{prop:upper}
  Let $u\in H_0^1(\Omega)$ be the solution of \eqref{eq:weak} and $\uG\in\VG$ its respective \DGFEM approximation \eqref{ipdg} on the grid $\grid\in\grids$. Then,
  \begin{align*}
    \enorm[\grid]{u- \uG}^2\Cleq \sum_{\elm\in\grid}\est_\grid(\elm)^2.
  \end{align*}
  % where the constant in $\Cleq$ depends only on the shape regularity
  % of $\grid_0$, $d$, and on $r$.
\end{prop}

The efficiency of the estimator follows with the standard bubble
function technique of Verf\"urth~\cite{Verfuerth:96,Verfuerth:2013};
compare also with \cite[Theorem~3.2]{KarakashianPascal:03},
\cite[Lemma 4.1]{Gudi:10} and Proposition~\ref{prop:lower8} below.
\begin{prop}\label{prop:lower}
  Let $u\in H_0^1(\Omega)$ be the solution of \eqref{eq:weak}
  and let $\grid\in\grids$. Then, for all $v\in\VG$ and $\elm\in\grid$, we have 
  \begin{multline*}
    \int_{\elm}\hG^2|f+\Delta
      v|^2\dx+\int_{\partial\elm\cap\Omega} \hG\jump{\nabla
  v}^2\ds\\
    \Cleq \norm[\omegaG(\elm)]{u-v}^2+\norm[\omegaG(\elm)]{\nablaG (u-v)}^2+\osc(\neighG(\elm),f)^2,
  \end{multline*}
with data-oscillation defined by
  \begin{align*}
    \osc(\mathcal{M},f):=\Big(\sum_{\elm'\in\mathcal{M}}\osc(\elm,f)^2\Big)^{1/2},\quad\text{where}\quad\osc(\elm,f):=\inf_{f_{\elm}\in\P_{r-1}}\norm[\elm]{\hG (f-f_{\elm})},
  \end{align*}
  for all $\mathcal{M}\subset\grid$. In particular, this implies  
  \begin{align*}
    \est_\grid(v,\elm)\Cleq \enorm[\neighG(\elm)]{v-u}+\osc(\neighG(\elm),f).
  \end{align*}
  % The hidden constants depend only on the shape regularity
  % of $\grid_0$, $d$, and on $r$.
\end{prop}
% 
% \begin{rem}
%   In \cite{KarakashianPascal:03,BonitoNochetto:10} 
%   the term $\int_{\partial\elm}\hG^{-1}\jump{v}^2\ds$ in $\est_\grid(\elm)$ is scaled depending on
%   the penalty parameter $\sigma$. We emphasise that the results of
%   this paper hold independently of such a scaling. Therefore, we
%   neglect the scaling for the sake of simplicity and since 
%   such a scaling may indicate that constants like in the upper
%   bound of Proposition~\ref{prop:upper} are independent of $\sigma$,
%   which is not the case.
% \end{rem}
% 
\begin{rem}
    Note that the presented theory obviously applies to all locally
    equivalent estimators as well; compare e.g. with
    \cite{KarakashianPascal:03,BonitoNochetto:10,BustinzaGaticaCockburn:2005,BeckerHansboLarson:2003,CarstensenGudiJensen:09}. For
    the sake of a unified 
    presentation, we restrict ourselves to the above representation. 
\end{rem}

\subsection{Adaptive discontinuous Galerin finite element method (\ADGM)}
\label{sec:ADGFEM}

The adaptive algorithm, whose convergence will be shown below, reads as follows.
\begin{algo}[\ADGM]\label{algo:ADGFEM} Starting from an initial triangulation $\grid_0$,
  the adaptive algorithm is an iteration of the following form 

  \centering
  \begin{minipage}{0.6\linewidth}
    \begin{enumerate}
    \item $\uk=% \hG[\gridk]=
      \SOLVE(\VG[\grid_k])$;
    \item
      $\{\est_k(\elm)\}_{\elm\in\grid_k}=\ESTIMATE(\uk,\gridk)$;
    \item
      $\markedk=\MARK\big(\{\est_k(\elm)\}_{\elm\in\grid_k},\gridk\big)$;
    \item $\gridk[k+1]=\REFINE(\gridk,\markedk)$; increment $k$.
    \end{enumerate}
  \end{minipage}
\end{algo}
Here we have used the notation $\est_k(\elm):=\est_{\gridk}(\elm)$, for brevity.

\paragraph{\SOLVE} We assume that the output 
\begin{align*}
  \uG=\SOLVE(\VG)
\end{align*}
is the \DGFEM approximation~\eqref{ipdg} of $u$ with respect to $\VG$. 

\paragraph{\ESTIMATE} We suppose that 
\begin{align*}
  \{\est_\grid(\elm)\}_{\elm\in\grid}:=\ESTIMATE(\uG,\grid)
\end{align*}
computes the error indicators from Section~\ref{sec:aposteriori}.

\paragraph{\MARK}
We assume that the output 
\begin{align*}
  \mathcal{M}:=\MARK(\{\est_\grid(\elm)\}_{\elm\in\grid},\grid)
\end{align*}
of marked elements satisfies 
\begin{align}\label{eq:mark}
  \est_\grid(\elm)\le g(\est_\grid(\mathcal{M})), \qquad\text{for
  all}~\elm\in\grid\setminus\mathcal{M}.
\end{align}
Here $g:\R_0^+\to\R_0^+$ is a fixed function, which is continuous in
$0$ with $g(0)=0$, i.e. $\lim_{\epsilon\to 0}g(\epsilon)=0$.

\paragraph{\REFINE}
We assume for $\mathcal{M}\subset\grid\in\grids$, that for the refined grid
\begin{align*}
  \tilde\grid:=\REFINE(\grid,\mathcal{M})
\end{align*}
we have 
\begin{align}\label{eq:refine}
  \elm \in\mathcal{M}\quad\Rightarrow \quad \elm\in \grid\setminus\tilde\grid,
\end{align}
i.e., each marked element is refined at least once. 

\subsection{The main result}
\label{sec:main-result}

The main result of this work states that the sequence of
discontinuous Galerkin approxiations, produced by \ADGM, converges to
the exact solution of~\eqref{eq:elliptic}.
\begin{thm}\label{thm:main}
  We have that 
  \begin{align*}
    \est_k(\gridk)\to 0 \quad\text{as}~k\to\infty.
  \end{align*}
  In particular, this implies that 
  \begin{align*}
    \enorm[k]{u-\uk}\to 0\quad\text{as}~k\to\infty.
  \end{align*}

\end{thm}

\section{A limit space and quasi-interpolation}
\label{sec:limit}
In this section we shall first introduce a new limit space $\V_\infty$ of the sequence of
adaptively constructed discontinuous finite element spaces
$\{\V(\gridk)\}_{k\in\N}$. A new
quasi-interpolation operator is then introduced in
Section~\ref{sec:quasi-ipol} in order to 
to prove that there exists a unique Galerkin solution $u_\infty$ of a
generalised discontinuous Galerkin problem in $\V_\infty$.

\subsection{Sequence of partitions}
\label{sec:prop-gridk}
The \ADGM produces a sequence $\{\gridk\}_{k\in\N_0}$ of nested
admissible partitions of $\Omega$. Following
\cite{MorinSiebertVeeser:08}, we define
\begin{align*}
  \grid^+:=\bigcup_{k\ge0}\bigcap_{j\ge
  k}\gridk[j],\qquad\text{and}\qquad\Omega^+:=\Omega(\grid^+)
\end{align*}
to be the set and domain of all elements, respectively, which eventually will not be refined any
more; here
$\Omega(X):=\operatorname{interior}\left(\bigcup\{\elm:\elm\in
  X\}\right)$ for a collection of elements $X$. 
We also define the complementary domain
$\Omega^-\definedas\Omega\setminus\Omega^+$. 
For the ease of presentation, in what follows, we shall
replace  subscripts $\gridk$ by $k$ to indicate the underlying
triangulation, e.g. we write $\neighk(\elm)$ instead of
$\neighG[\gridk](\elm)$.

The following result states that neighbours of elements in $\grid^+$
are eventually also elements of $\grid^+$; cf.,
\cite[Lemma 4.1]{MorinSiebertVeeser:08}. 
\begin{lem}\label{L:Nk=NK}
  For $\elm\in\grid^+$ there exists a constant $K=K(\elm)\in\N_0$, such that 
  \begin{align*}
    \neighk(\elm)=\neighk[K](\elm)\qquad\text{for all}~ k\ge K,
  \end{align*}
  i.e., we have $\neighk(\elm)\subset\grid^+$ for all $k\ge K$.
\end{lem}

 \begin{figure}
   %\begin{center}
   %\hspace*{-0.5cm} 
   {\resizebox{0.33\textwidth}{!}{\includegraphics{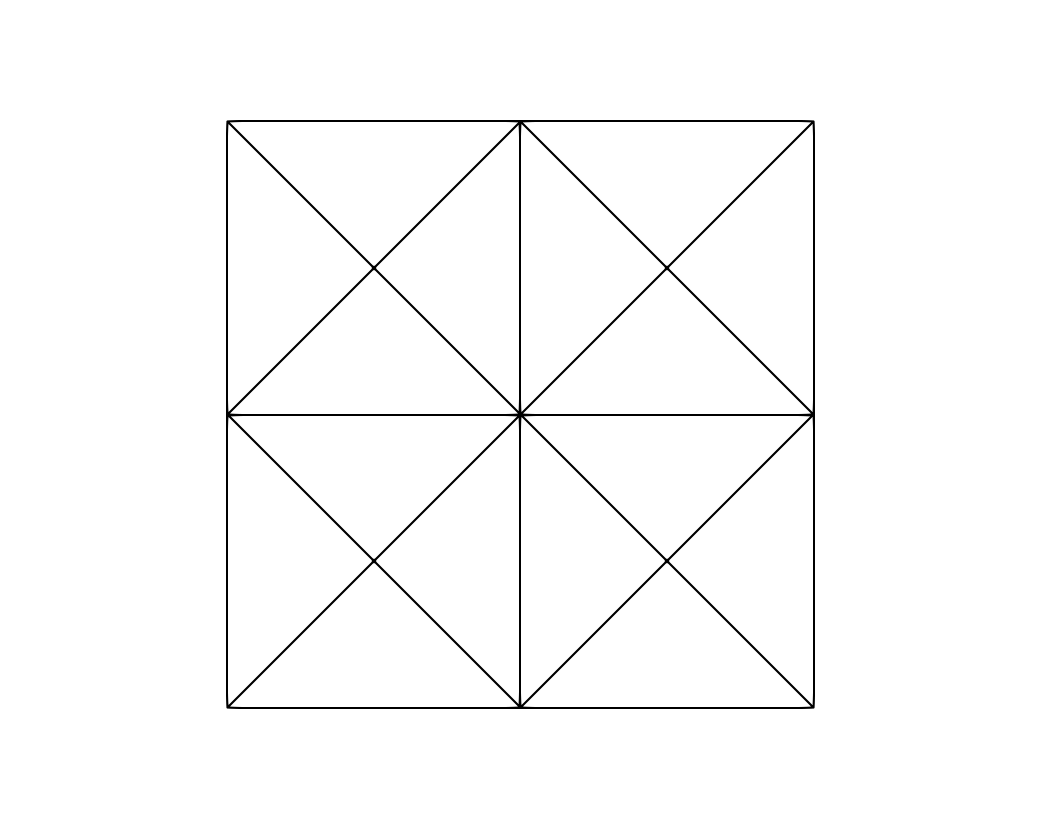}}}%
     %\hspace*{-1cm} 
     {\resizebox{0.33\textwidth}{!}{\includegraphics{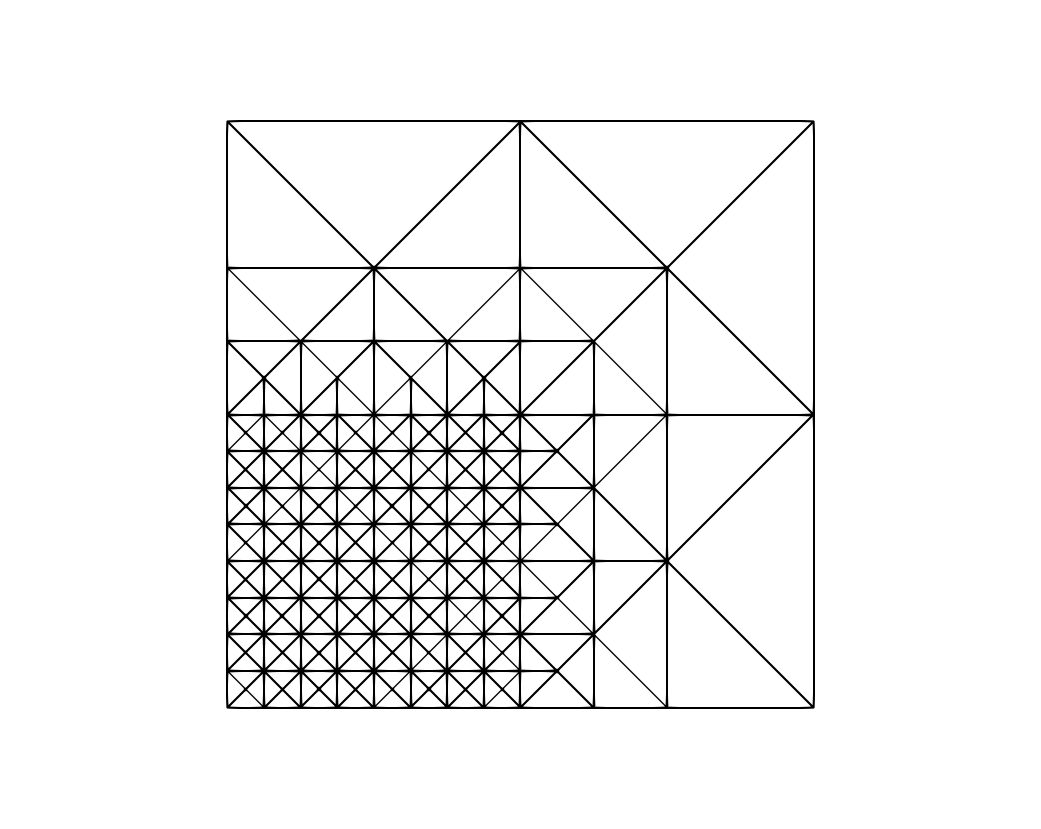}}}%
          {\resizebox{0.33\textwidth}{!}{\includegraphics{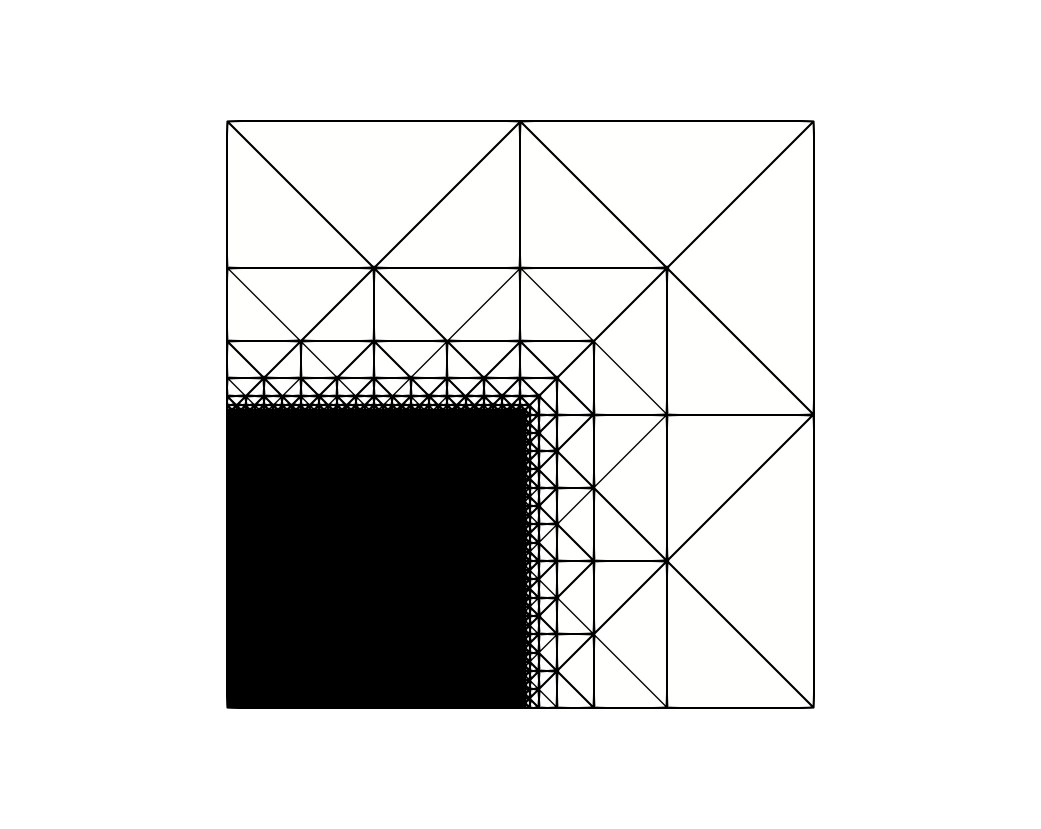}}}%
   %\end{center}
    \caption{Selection of a sequence of triangulations of
      $\Omega=(0,1)^2$, where in each iteration 
      the elements in $\Omega^{-}=[0,0.5]\times [0,0.5]$ are marked
      for refinement. The elements $\grid^+$ in the remaining domain
      $\Omega\setminus\Omega^-$ are, after some
      iteration, not refined anymore. Moreover, after some iteration,
      their whole neighbourhood is not refined anymore.}
     \label{fig:meshSeq}
 \end{figure}
Next, for a fixed $k\in\N_0$, we set 
  \begin{alignat*}{2}
    % \gridk^-&:=\{\elm\in\gridk\colon
    % \omegak(\elm)\subset\overline{\Omega^-} \},
    % &\qquad\Omega^-_k&:=\Omega(\gridk^-), 
    % \\
    \gridk^+&:=\gridk\cap\grid^+,&\qquad\Omega_k^{+}&:=\Omega(\gridk^{+}),
    \\
    \gridk^{++}&:=\{\elm\in\gridk\colon \neighk(\elm)\subset\grid^+\},&\qquad\Omega_k^{++}&:=\Omega(\gridk^{++}),
    \\
    \gridk^-&:=\gridk\setminus% (
    \gridk^{++}% \cup\gridk^-),
    &\qquad\Omega_k^-&:=\Omega(\gridk^-);
  \end{alignat*}
  compare also with Figure~\ref{fig:meshSeq}.
  This notation is also adopted for the corresponding faces, e.g., we
  denote $\sides_k^+:=\sides(\gridk^+)$ and
  $\mathring\sides_k^+:=\mathring\sides(\gridk^+)$ and correspondingly
  for all other above sub-triangulations of $\gridk$.

  The next lemma states that the meshsize of elements in
   $\gridk^-$ converges uniformly to zero; compare also
   with~\cite[(4.15) and Corollary 4.1]{MorinSiebertVeeser:08} and  \cite[Corollary 3.3]{Siebert:11}.% on $\Omega_k^-$ the mis proved in related to 
  % \cite[(4.15) and Corollary 4.1]{MorinSiebertVeeser:08}. However, the 
  % definition of $\Omega^-_k$ % and $\Omega_k^-$
  % differs from the corresponding one in~\cite{MorinSiebertVeeser:08},
  % which requires some modifications in the proof.}
  \begin{lem}\label{lem:Omegastar}
      We have that 
       $
      \lim_{k\to\infty}\|h_k\chi_{\Omega_k^-}\|_{L^\infty(\Omega)}
      =0$, with $\chi_{\Omega_k^-}$ denoting the characteristic
      function of $\Omega_k^-$.

      In particular, this implies that $
       \lim_{k\to\infty}|\Omega(\neighk(\grid^-_k))\setminus\Omega^-|=
       \lim_{k\to\infty}|\Omega_k^-\setminus\Omega^-|=0$. 
    \end{lem}

    \begin{proof}
      The first claim is proved in~\cite[Corollary 3.3]{Siebert:11}.

      In order to prove the second claim, we first observe that
      \begin{align*}
        |\Omega_k^-\setminus\Omega^-|=|\Omega(\grid^+\setminus\gridk^{++})|.
      \end{align*}
      For $\ell\in\N$, it 
      follows from Lemma~\ref{L:Nk=NK} and $\#\gridk[\ell]^{+}<\infty$, 
      that there exists $K(\ell)$, such that
      $\gridk[\ell]^{+}\subset\gridk[K(\ell)]^{++}$ and thus
      \begin{align}\label{Gk+-Gk++}
        |\Omega(\grid^+\setminus \grid_{K(\ell)}^{++})|\le |
       \Omega(\grid^+\setminus\grid_\ell^+)|\to 0\qquad\text{as}~\ell\to\infty,
      \end{align}
      i.e. a subsequence of
      $\{|\Omega(\grid^+\setminus\gridk^{++})|\}_k$ vanishes. However,
      since the sequence is monotone decreasing, it must vanish as a whole. To conclude, we first realise that
      $|\Omega(\neighk(\grid^-_k))\setminus\Omega^-|\le
     |\Omega_k^-\setminus\Omega^-|+
     |\Omega(\neighk(\grid^-_k)\setminus\gridk^-)|$; it remains to
     prove that the latter term vanishes as $k\to\infty$. To this end, we observe that
      \begin{align*}
        \neighk(\grid^-_k)\setminus\gridk^-\subset\gridk^{++}\setminus\gridk^{3+},
      \end{align*}
      with
      $\gridk^{3+}=\{\elm\in\gridk\colon\neighk(\elm)\in\gridk^{++}\}$. Indeed,
      assume that $\elm\in\gridk^{3+}\cap\neighk(\gridk^-)$, then
      there exists $\elm'\in\neighk(\elm)$ with
      $\elm'\in\gridk^{++}\cap\gridk^-$; this is a
      contradiction. Consequently, we have
      \begin{align*}
        |\Omega(\neighk(\grid^-_k)\setminus\gridk^-)|\le
        |\Omega(\gridk^{++}\setminus\gridk^{3+})|\le
        |\Omega(\grid^{+}\setminus\gridk^{3+})|\to
        0\quad\text{as}~k\to\infty, 
      \end{align*}
      where the last limit follows by iterating the reasoning from~\eqref{Gk+-Gk++}. 
    \end{proof}

\subsection{The limit space}
\label{sec:space-limit}
In this section, we shall investigate the limit of the finite element
spaces $\V_k:=\V(\gridk)$, $k\in\N$. To this end, we
define
\begin{align*}
  \V_\infty:=\big\{v\in BV(\Omega)&:
                                    v|_{\Omega^-}\in
                                    H^1_{\partial\Omega\cap\partial\Omega^-}(\Omega^-)~\text{and}~v|_\elm\in\P_r~\forall\elm\in \grid^+
  \\
                                  &\quad\text{such that}~
                                    \exists \{v_k\}_{k\in\N}, v_k\in\V_k
                                    ~\text{with}~ \limsup_{k\to\infty}\enorm[k]{v_k}<\infty
  \\
                                  &\quad\text{and}~
                                   \lim_{k\to\infty}\enorm[k]{v-v_k}+\norm[\Omega]{v-v_k}=0
                                    \big\};
\end{align*}
here $H^1_{\partial\Omega\cap\partial\Omega^-}(\Omega^-)$ denotes the
space of functions from $H_0^1(\Omega)$ restricted to
$\Omega^-$. Moreover, we have extended the definition of the
piece-wise gradient to
\begin{align}\label{df:nablapw}
  \nablaG v\in L^2(\Omega)^d:\quad \nablaG v|_{\Omega^-}=\nabla
  v|_{\Omega^-}\quad\text{and}\quad \nablaG v|_\elm=\nabla
  v|_\elm~\forall \elm\in\grid^+.
\end{align}
Note that for $v\in BV(\Omega)$ there exists the $L^1$-trace of $v$ on
$\Gamma_k=\bigcup\{S:S\in\sides_k\}$; compare e.g. with the trace
theorem \cite[Theorem 4.2]{BuffaOrtner:09}. In other words, $v$ is
measurable with respect to the $(d-1)$-dimensional Hausdorff measure on
$\sides_k$ and, therefore, the term $\enorm[k]{v}$, $v\in\V_\infty$, makes sense. 
Obviously, we have $\V_k\cap C(\Omega)\subset \V_\infty$ for all $k\in
\N$ and, thus,  $\V_\infty$ is not empty.

Setting $h_+:=h_{\grid^+}$ and $\sides^+:=\sides(\grid^+)$, we define
\begin{align*}
  \scp[\infty]{v}{w}&:=\int_\Omega \nablaG v\cdot\nablaG w\dx% \int_{\Omega^-}\nabla
  % v\cdot\nabla w\dx +
  % \int_{\grid^+}\nabla v\cdot\nabla w\dx
  + \bar\sigma\int_{\sides^+}\hG[+]^{-1}\jump{v}\jump{w}\ds,
\end{align*}
and 
$  \enorm[\infty]{v}:=\scp[\infty]{v}{v}^{1/2}$,
for all $v,w\in\V_\infty$. % For brevity, we shall frequently use the notation
% \begin{gather*}
%   \int_{\Omega}\nablaG v\cdot\nablaG w\dx \equiv\int_{\Omega^-}\nabla
%   v\cdot\nabla w\dx +% \sum_{
% % \elm\in\grid^+}
% \int_{\grid^+}\nabla v\cdot\nabla w\dx.
% \end{gather*}

% In this section we shall investigate the limit of the finite element
% spaces $\V_k:=\V(\gridk)$, $k\in\N$. Tho this end, we define
% \begin{align*}
%   \V_\infty:=\big\{v\in L^2(\Omega)&: v|_{\Omega^-}\in
%   H^1_{\partial\Omega\cap\partial\Omega^-}(\Omega^-)~\text{and}~
%   v|_\elm\in \P_r~\forall \elm\in\grid^+\\
%   &\quad\text{with}~
%     \limsup_{k\to\infty}\enorm[k]{v}<\infty\\
%   &\quad
%     \text{and}~\exists \{v_k\}_{k\in\N}, v_k\in\V_k
%     ~\text{with}~\lim_{k\to\infty}\enorm[k]{v-v_k}=0
%   \big\}
% \end{align*}
% For $h_+:=h_{\grid^+}$ and $\sides^+:=\sides(\grid^+)$ let
% \begin{align*}
%   \scp[\infty]{v}{w}:=\int_\Omega vw\dx+\int_{\Omega}\nablaG v\cdot\nablaG w\dx 
%   + \int_{\sides^+}\hG[+]^{-1}\jump{v}\jump{w}\ds
% \end{align*}
% for all $v,w\in\V_\infty$.

% Note that the term $\int_\side \jump{v}^2\ds$ for $v\in\V_\infty$ and
% $S\in\sides_k$ is well posed.  This is clear for
% $S\cap(\Omega\setminus\overline{\Omega^-})$ and $S\cap\Omega^-$ since
% the former is a possibly infinite union of sides in $\sides^+$ and in
% the latter case we have $\jump{v}|_{S\cap\Omega^-}\equiv0$ from the trace theorem in
% $H^1(\Omega^-)$. For $S\cap\partial\Omega$ similar arguments apply and
% it remains to consider the case $S\cap\partial\Omega^-\cap\Omega$.
% {\Large\textbf{HERE TODO}}

We shall next list some basic properties of the space $\V_\infty$.

\begin{prop}\label{prop:V}
  For $v\in\V_\infty$, we have 
  \begin{align*}
    \enorm[k]{v}\nearrow\enorm[\infty]{v}<\infty% \qquad\text{and}\qquad
    % \enorm[k]{v-v_k}\to 0
    \quad\text{ as}~
    k\to\infty.
  \end{align*}
  In particular, for fixed $\ell\in\N$, let
  $\elm\in\gridk[\ell]$; then, we have
  \begin{align*}
    \int_{\{S\in\sides_k:S\subset\elm\}}h_k^{-1}\jump{v}^2\ds\nearrow \int_{\{S\in\sides^+:S\subset\elm\}}h_+^{-1}\jump{v}^2\ds,\quad\text{as}~k\to\infty.
  \end{align*}
\end{prop}

\begin{proof}
  % Similarly, it follows
  % $\norm[L^1(\Gamma_k)]{v-v_k}\Cleq \norm[BV(\Omega)]{v-v_k}\to 0$ as
  % $k\to \infty$. 
  % We observe
  % that $\Gamma_k\subset\Gamma_m$ for $m\ge k$ and thus, thanks to
  % $\limsup_{k\to\infty}\enorm[k]{v_k}<\infty$,  we have that 
  % $\{v_\ell|_{\Gamma_k}\}_{\ell\in\N}$ is bounded in $L^2(\Gamma_k)$
  % and thus there exists a weak limit in $L^2(\Gamma_m)$, which must
  % coincide with
  % $v|_{\Gamma_k}$, i.e. $v_\ell|_{\Gamma_k}\to v|_{\Gamma_k}\in
  % L^2(\Gamma_k)$ and thus $\enorm[k]{v-v_\ell}\to
  % 0$ as $\ell\to \infty$.
  Since $v\in\V_\infty$, there exists $\{v_k\}_{k\in\N}, v_k\in\V_k$
  with $\lim_{k\to\infty}\enorm[k]{v-v_k}=0$ and 
  $\limsup_{k\to\infty}\enorm[k]{v_k}<\infty$. 
  We first observe that 
  \begin{align*}
    \enorm[k]{v}\le \enorm[k]{v-v_k}+\enorm[k]{v_k}<\infty
  \end{align*}
  uniformly in $k$. 
  Thanks to the mesh-size reduction, i.e. $h_m\le h_k$ for all $m\ge
  k$, we conclude that
  \begin{align*}
      \int_{\sides_k}h_k^{-1}\jump{v}^2\ds \le
      \int_{\sides_k}h_m^{-1}\jump{v}^2\ds \le
      \int_{\sides_m}h_m^{-1}\jump{v}^2\ds, 
    \end{align*}
    thanks to the inclusion 
    $\bigcup_{\sides\in\sides_k}\side\subset
    \bigcup_{\sides\in\sides_m}\side$. 
  
  Therefore, we have 
  $\enorm[k]{v}\le\enorm[m]{v}$ for all
  $m\ge k$ and, thus, $\{\enorm[k]{v}\}_{k\in\N}$ converges.
  Consequently, for $\epsilon>0$ there exists $K=K(\epsilon)$, such
  that for all $k\ge K$ and $m>k$ large enough, we have
  \begin{align*}
    \epsilon>|\enorm[m]{v}^2-\enorm[k]{v}^2|&=\bar\sigma
                                                 \int_{\sides_m\setminus(\sides_m\cap\sides_k)}\hG[m]^{-1}\jump{v}^2\ds-\bar\sigma
                                                 \int_{\sides_k\setminus(\sides_m\cap\sides_k)}\hG[k]^{-1}\jump{v}^2\ds
    \\
                                               &\ge (2^{1/(d-1)}-1) \,\bar\sigma
                                                 \int_{\sides_k\setminus(\sides_m\cap\sides_k)}\hG[k]^{-1}\jump{v}^2\ds
                                                 \\
    &\ge (2^{1/(d-1)}-1) \,\bar\sigma
                                                 \int_{\sides_k\setminus\sides_k^+}\hG[k]^{-1}\jump{v}^2\ds.
  \end{align*}
  This follows from the fact
  that $\hG[m]|_\side \le 2^{-1/(d-1)}\hG[k]|_\side $ for all
  $S\in\sides_k\setminus(\sides_m\cap\sides_k)$
   together with 
  $\sides_k^+=\sides_m\cap\sides_k$ for sufficiently large $m>k$. 
  
  Therefore, we have 
  $\int_{\sides_k\setminus\sides_k^{+}}\hG[k]^{-1}\jump{v}^2\ds\to 0$
  as $k\to\infty$ and, thus, 
  \begin{align*}
    \enorm[k]{v}^2&= % \int_\Omega| v|^2\dx+\int_\Omega|\nablaG v|^2\dx +
             %       \bar\sigma\int_{\sides_k}h_k^{-1}\jump{v}^2\ds
%    \\
%                  & = 
                     \int_{\Omega} |\nablaG v|^2\dx 
                    + \bar\sigma\int_{\sides_k^+}h_k^{-1}\jump{v}^2\ds
                    +\bar\sigma\int_{\sides_k\setminus\sides_k^+}h_k^{-1}\jump{v}^2\ds
                     \to \enorm[\infty]{v}^2+0.
  \end{align*}
  This proves the first claim. The second claim is a localised version
  and follows completely analogously. 
\end{proof}

As a consequence, we have that Friedrichs and Poincar\'e inequalities
are inherited to \(\V_\infty\).
\begin{cor}[Friedrichs-\(\V_\infty\)]\label{C:FriedrichsVinfty}
  We have
  \begin{align*}
    \norm[\Omega]{v} \Cleq \enorm[\infty]{v}\quad\text{for all}~v\in\V_\infty.
  \end{align*}
\end{cor}
\begin{proof}
Since $v\in\V_\infty$, there exists $\{v_k\}_{k\in\N}, v_k\in\V_k$
  with $\limsup_{k\to\infty}\enorm[k]{v_k}<\infty$ and
  $\lim_{k\to\infty}\enorm[k]{v-v_k}+\norm[\Omega]{v-v_k}=0$. It thus
  follows from Corollary~\ref{C:Friedrichs} and
  Proposition~\ref{prop:V} that
  \begin{align*}
    \norm[\Omega]{v}=\lim_{k\to\infty}\norm[\Omega]{v_k}\Cleq\lim_{k\to\infty}\enorm[k]{v_k}=\enorm[\infty]{v}.
  \end{align*}
  This is the assertion.
\end{proof}

\begin{lem}[Poincar\'e-$\V_\infty$]\label{lem:poincare}
  Fix $k\in\N$ and let $\elm\in \gridk$. Then for $v\in\V_\infty$ and
  $v_\elm:=\frac1{|\omegak(\elm)|}\int_{\omegak(\elm)}v\dx$, we have 
  \begin{align*}
    \norm[\omegak(\elm)]{v-v_\elm}^2\Cleq \norm[\omegak(\elm)]{h_k\nablaG
    v}^2+\int_{\{S\in\sides^+:S\subset \omegak(\elm)\}}h_k^2h_+^{-1}\jump{v}^2\ds.
  \end{align*}
\end{lem}

\begin{proof}
  By the definition of $\V_\infty$, there exists 
  $v_\ell\in\V_\ell$, $\ell\in\N_0$, with
  $\lim_{\ell\to\infty}\enorm[\ell]{v-v_\ell}+\norm[\Omega]{v-v_\ell}=0$ and
  $\limsup_{\ell\to\infty}\enorm[\ell]{v_\ell}<\infty$.  
  Therefore, we have 
  \begin{multline*}
    \norm[\omegak(\elm)]{\nablaG
    v_\ell}^2+\int_{\{S\in\sides_\ell:S\subset
    \omegak(\elm)\}}h_\ell^{-1}\jump{v_\ell}^2\ds
    \\
    \to \norm[\omegak(\elm)]{\nablaG
    v}^2+\int_{\{S\in\sides^+:S\subset \omegak(\elm)\}}h_+^{-1}\jump{v}^2\ds\quad\text{as}~\ell\to\infty;
  \end{multline*}
  see Proposition~\ref{prop:V}. Moreover, we have 
  \begin{align*}
    \norm[\omega_k(\elm)]{v_\elm-v_{\ell,\elm}}\le
    \norm[\omega_k(\elm)]{v-v_{\ell}}\le \norm[\Omega]{v-v_\ell}\to 0\quad\text{as}~\ell\to\infty,
  \end{align*}
  where
  $v_{\ell,\elm}:=\frac1{|\omegak(\elm)|}\int_{\omegak(\elm)}v_\ell\dx$. 
  We conclude with
  Proposition~\ref{prop:poincareG} that
  \begin{align*}
    \norm[\omegak(\elm)]{v-v_\elm}^2&\leftarrow
    \norm[\omegak(\elm)]{v_\ell-v_{\ell,\elm}}^2
    \\
    &\Cleq \norm[\omegak(\elm)]{h_k\nablaG
    v_\ell}^2+\int_{\{S\in\sides_\ell:S\subset
      \omegak(\elm)\}}h_k^2h_\ell^{-1}\jump{v_\ell}^2\ds
    \\
    &\to \norm[\omegak(\elm)]{h_k\nablaG
    v}^2+\int_{\{S\in\sides^+:S\subset \omegak(\elm)\}}h_k^2h_+^{-1}\jump{v}^2\ds,
  \end{align*}
  as $\ell\to\infty$.
\end{proof}

In order to extend the dG bilinear form~\eqref{dgbilinear} to $\V_\infty$, we need
  to define appropriate lifting operators.  For each $\side\in\sides^+$,
  there exists $\ell=\ell(\side)\in\N$, such that $\side\in\sides_\ell^{++}$% and 
  % we have for all $\side\subset\elm\in\grid^+$ that $\elm\in\gridk[\ell]^+$
  .
  We define
  the local lifting operators
  $\riftS[\infty]:L^2(S)^d\to L^2(\Omega)^d$ and 
  $\liftS[\infty]:L^2(S)\to L^2(\Omega)^d$ by
  \begin{align}\label{eq:lift-inf}
    \riftS[\infty]=\riftS[\ell]:=\riftS[{\gridk[\ell]}]\qquad
                     \text{and}\qquad
                     \liftS[\infty]=\liftS[\ell]:=\liftS[{\gridk[\ell]}].
    \end{align}
    From~\eqref{eq:liftG} it is easy to see, that 
    $\riftS[\ell]$ and $\liftS[\ell]$ depend only on $S$ and the
    at most two adjacent elements $\elm,\elm'\in \gridk[\ell]^+$ with $\side\subset\elm\cap\elm'$.
    Therefore, and thanks to the fact that 
    the $\gridk^+$ are nested, we have that $\riftS[\ell]=\riftS[k]$
    for all $k\ge \ell$ and, thus,
    the definition is unique.  
 % To this end let 
 %  \begin{align*}
 %    \mathbb{L}_\infty:=\{v\in L^2(\Omega): v|_\elm \in\P_r~\forall\elm\in\grid^+\}.
 %  \end{align*}
 %  Obviously, $\mathbb{L}_\infty$ is a closed subspace of
 %  $L^2(\Omega)$ and $\V_\infty\subset\mathbb{L}_\infty$. 
 %  Thus, for $\side\in\sides^+$ the local lifting operators
 %  $\riftS[\infty]:L^2(S)^d\to \mathbb{L}_\infty^d$ and 
 %  $\liftS[\infty]:L^2(S)\to \mathbb{L}_\infty^d$ defined by
 %  \begin{align*}
 %      \int_\Omega \riftS[\infty](\vec{\phi})\cdot
 %      \vec{\tau}\dx&=\int_\side \vec{\phi}\cdot\mean{\vec{\tau}}\ds\qquad\forall
 %                     \vec{\tau}\in \mathbb{L}_\infty^d
 %                     \intertext{and}
 %                     \int_\Omega \liftS[\infty](q)\cdot
 %                     \vec{\tau}\dx&=\int_\side q\jump{\vec{\tau}}\ds\qquad\forall
 %                                    \vec{\tau}\in \mathbb{L}_\infty^d
 %    \end{align*}
 %    are well posed.
    We formally define the global lifting operators % $\riftG[\infty]:L^2(\Gamma^+)^d\to L^2(\Omega)^d$ and
    % $\liftG[\infty]:L^2(\mathring\Gamma^+)\to L^2(\Omega)^d$
    by
    % \begin{align*}
    %   \riftG[\infty](\vec{\phi}):=\sum_{\side\in\sides^+}\riftS[\infty](\vec{\phi})\qquad\text{and}\qquad \liftG[\infty](q):=\sum_{\side\in\mathring\sides^+}\riftS[\infty](q);
    % \end{align*}
    \begin{align*}
      \riftG[\infty]:=\sum_{\side\in\sides^+}\riftS[\infty]\qquad\text{and}\qquad \liftG[\infty]:=\sum_{\side\in\mathring\sides^+}\liftS[\infty];
    \end{align*}
    here
    $\mathring\sides^+:=\{\side\in\sides^+:\side\not\in\partial\Omega\}$.

Moreover, from the local estimates 
  \eqref{eq:liftGstab}, it is easy to see that for $v\in\V_\infty$ and
  $\vec{\beta}\in \R^d$, we have that 
  $\sum_{\side\in\sides_k^+}\riftS[\infty](\jump{v})$ and
  $\sum_{\side\in\mathring\sides_k^+}\liftS[\infty](\vec{\beta}\cdot\jump{v})$
  are Cauchy sequences in $L^2(\Omega)^d$. 
  Consequently,
  $\riftG[\infty](\jump{v}),\liftG[\infty](\vec{\beta}\cdot\jump{v})\in
  L^2(\Omega)$ are well posed 
  and we have 
  \begin{align}\label{eq:lift-stab}
      \norm{\riftG[\infty](\jump{v})}\Cleq 
      \norm[\Gamma^+]{\hG[+]^{-1/2} v}
      \quad\text{and}\quad 
      \norm{\liftG[\infty](\vec{\beta}\cdot\jump{v})}\Cleq \abs{\vec{\beta}}
      \norm[\mathring\Gamma^+]{\hG[+]^{-1/2} v},
    \end{align}
    where
    $\Gamma^+=\bigcup\{\side:\side\in\sides^+\}$ and
    $\mathring\Gamma^+=\bigcup\{\side:\side\in\mathring\sides^{+}\}$. 
    This enables us to generalise the discontinuous Galerkin bilinear form   
    to $\V_\infty$ setting
    \begin{align*}
    \bilin[\infty]{w}{v}&:=\int_{\Omega}\nablaG w\cdot\nablaG v\,\ud x -\int_{\sides^+}\big(\mean{\nabla w}\cdot\jump{v}+\theta \mean{\nabla v}\cdot\jump{w}\big)\ds
       \\
      &\quad +\int_{\mathring\sides^+}
       \big(\vec{\beta}\cdot\jump{w}\jump{\nabla v}+\jump{\nabla
         w}\vec{\beta}\cdot\jump{v}\big)\ds
       \\
       &\quad + \int_\Omega \gamma \big(\riftG[\infty](\jump{w})+\liftG[\infty](\vec{\beta}\cdot\jump{w})\big)\cdot \big(\riftG[\infty](\jump{v})+\liftG[\infty](\vec{\beta}\cdot\jump{v})\big)\dx
       \\
                          &\quad+\int_{\sides^+}\frac{\sigma}{h_+}\jump{w}\cdot\jump{v}\,\ud s,
 \end{align*}
 for $v,w\in\V_\infty$.

\begin{lem}\label{lem:V}
  The space $\big(\V_\infty,\scp[\infty]{\cdot}{\cdot}\big)$ is a
  Hilbert space.
\end{lem}

\begin{cor}\label{cor:u8}
  There exists a unique $u_\infty\in\V_\infty$, such that 
  \begin{align}\label{eq:u8}
    \bilin[\infty]{u_\infty}{v}=\int_\Omega fv\dx\qquad\text{for
    all}~v\in\V_\infty.
  \end{align}
 
\end{cor}

In order to prove the last two statements, we introduce
a new quasi-interpolation, which is
designed  in due consideration of the future 
refinements. The proofs of Lemma~\ref{lem:V} and
Corollary~\ref{cor:u8} are postponed to the end of
Section~\ref{sec:quasi-ipol}.

\subsection{Quasi-interpolation}
\label{sec:quasi-ipol}

We shall now define a quasi-interpolation
operator $\Pi_k$, which maps into $\V_\infty\cap\V_k$; this will be a
key technical tool in the analysis. On the one 
hand, membership in $\V_\infty\cap\V_k$ suggests to use some Cl\'ement
type interpolation since the mapped
functions need to be continuous in $\Omega^-$.
On the other hand, the fact that the
\ADGM may leave some elements (namely $\gridk^+\supset\gridk^{++}$) unrefined, suggests
to define $\Pi_k$ to be the identity on these elements. Note that the
quasi-interpolation operator
from~\cite{CarstensenGallistlSchedensack:2013} 
is motivated by a similar idea in order to map from one Crouzeix-Raviart
space into its intersection with a finer one.

For fixed $k\in\N$, let 
$\{\Phi_z^\elm:\elm\in\gridk,~z\in \nodes_k(\elm)\}$
be the Lagrange basis of $\Vk:=\VG[\gridk]$, i.e., $\Phi_z^\elm$ is a piecewise polynomial
of degree $r$ with 
$\supp(\Phi_z^\elm) =\elm$ and 
\begin{align*}
  \Phi_z^\elm(y)=\delta_{zy}\qquad\text{for
  all}~z,y\in\nodes_k.
\end{align*}
Its dual basis is then the set $\{\Psi_z^\elm:\elm\in\gridk,~z\in
\nodes_k(\elm)\}$ of piecewise polynomials
of degree $r$, such that 
$\supp(\Psi_z^\elm) =\elm$ and
\begin{align*}
  \scp{\Psi_y^\elm}{\Phi_z^\elm}=\delta_{zy}\qquad\text{for
  all}~z,y\in\nodes_k(\elm).
\end{align*}
For all $\ell\ge k$, we define $\Pi_k:L^1(\Omega)\to L^1(\Omega)$ by 
\begin{align}\label{df:Pik}
  \Pi_kv:=\sum_{\elm\in\gridk}\sum_{z\in\nodes_k(\elm)}
  (\Pi_kv)|_\elm(z) \,\Phi_z^\elm,
\end{align}
where for $z\in\nodes_k(\elm)$ we have that 
\begin{align}\label{df:Pikz}
  (\Pi_kv)|_\elm(z):=
  \begin{cases}
    \int_\elm
    v\Psi_z^\elm\dx,\quad&\text{if}~% \exists\elm'\in
    \neighk(z)\cap\gridk^{++}\neq \emptyset
    \\
    0,\quad&\text{else if}~z\in\partial\Omega
    \\
    \sum_{\elm'\in\neighk(z)}\frac{|\elm'|}{|\omegak(z)|}\int_{\elm'}v\Psi_z^{\elm'}\dx,\quad&\text{else.}
  \end{cases}
\end{align}
Beyond standard stability and interpolation estimates for
$H_0^1(\Omega)$ functions~\cite{ScottZhang:90,DemlowGeorgoulis:12}, we
list the following properties related to our setting.
% \todoin{What happens when $z$ is a boundary node and also in $++$ set?
%   Does it matter? Doesn't seem to matter, but just checking...\\
% Answer: 
% In the $++$ set we want have identity (i.e. (4)) and non-conformity is
% allowed. In the rest of the set we want to have a conforming
% interpolation and therefore zero boundary values have to be applied.}

\begin{lem}[Properties of $\Pi_k$]\label{lem:Pik}
  The operator $\Pi_k:L^1(\Omega)\to L^1(\Omega)$ defined
  in~\eqref{df:Pik} has the following properties:
  \begin{enumerate}
  \item\label{Pik:1} $\Pi_k:L^p(\Omega)\to L^p(\Omega)$ is a linear and
    bounded projection for all $1\le p\le\infty$. In particular, we have that 
    \begin{align*}
      \norm[L^p(\elm)]{\Pi_kv}\Cleq \norm[L^p(\omegak(\elm))]{v},
    \end{align*}
    where the constant solely depends on $p$, $r$, $d$, and  the shape regularity
    of $\grid_0$. 
  \item\label{Pik:2} $\Pi_kv\in\Vk$ for all $v\in L^1(\Omega)$;
     \item\label{Pik:5} $\Pi_kv|_\elm=v|_\elm$, if $\elm\in\gridk$ and
    $v|_{\omegak(\elm)}\in\P_r(\omega_k(\elm))$;
    % \todoin{Do we really mean $\P_r(\omega_k(\elm))$ here, i.e., one
    %   polynomial over $\omega_k(\elm)$? It looks we may want to say
    %   piecewise polynomial\\
    %   Answer:
    %   The claim holds for all $\elm\in\gridk$ and thus the properties of
    %   Scott-Zhang interpolation apply, i.e. in the third case of (3.3) the identity
    %   is only guaranteed, if
    %   $v|_{\omegak(\elm)}\in\P_r(\omega_k(\elm))$. In the $++$ set
    %   (see (4)) it is more local since we do not have a coupling
    %   condition (continuity) between elements.
    % }
  \item\label{Pik:4} $\Pi_kv|_\elm=v|_\elm$, if $\elm\in\gridk^{++}$ and
    $v|_\elm\in\P_r(\elm)$; if moreover $v\in\Vk$, then also
    $\jump{v-\Pi_kv}|_S\equiv0$ for all $\side\in\sides_k^{++}$. 
  \item\label{Pik:3} $\Pi_kv|_{\Omega\setminus\Omega_k^{+}}\in
    C(\overline{\Omega\setminus\Omega_k^{+}})$ and $\jump{\Pi_kv
      }=0$ on $\partial(\Omega\setminus\Omega^+_k)$;
  \item \label{Pik:4a} $\Pi_k v=v$, for all $v\in\V_k$ with $v|_{\Omega\setminus\Omega_k^{++}}\in
    C(\Omega\setminus\Omega_k^{++})$;
  \item \label{Pik:5a} $\Pi_kv\in \V_\infty$, and we have $\enorm[k]{\Pi_kv}=\enorm[\infty]{\Pi_kv}$.
  \end{enumerate}
\end{lem}

\begin{proof}
  Claims~\eqref{Pik:1}--\eqref{Pik:5} follow by standard estimates for
  the Scott-Zhang operator~\cite{ScottZhang:90,DemlowGeorgoulis:12}. 

  Assertion~\eqref{Pik:4} is a consequence of the
  definition~\eqref{df:Pikz} of $\Pi_k$ since $\elm\in\gridk^{++}$
  implies that $\neighk(\elm)\cap\gridk^{++}=\neighk(\elm)$.
  Note that $v\in\VG$ implies $v|_\elm\in\P_r(\elm)$ for all
  $\elm\in\gridk$ and thus $(\Pi_k v)|_\elm(z)=v|_\elm(z)$ for all
  $\elm\in\neighk(z)$ if
  $\neighk(z)\cap\gridk^{++}\neq\emptyset$. This is in particular the
  case when $z\in \side\cap\nodes_k$ with $\side\in\sides_k^{++}$.

  For $\elm\in\gridk\setminus\gridk^+$, we have that
  $\neighk(z)\cap\gridk^{++}=\emptyset$ since otherwise there exists
  $\elm'\in \neighk(\elm)\cap\gridk^{++}$ and thus $\elm\in
  \neighk(\elm')$, which implies $\elm\in\gridk^+$, thanks to the
  definition of $\gridk^{++}$. Therefore, \eqref{df:Pikz} implies    
  that $\Pi_kv$ is continuous on 
  $\Omega\setminus\Omega_k^+$. Moreover, for $z\in
    \nodes_k(\elm)\cap\Omega\setminus\Omega^+_k$, definition~\eqref{df:Pikz} is
    independent of $\elm$ and thus $\Pi_kv$ does not jump across the
    boundary $\Omega\setminus\Omega_k^+$.  
    This completes the proof of~\eqref{Pik:3}.

  On the one hand, if $v\in\V_k$ with $v|_{\Omega\setminus\Omega_k^+}\in
  C(\overline{\Omega\setminus\Omega_k^{++}})$ then we have clearly
  $\Pi_kv|_{\Omega\setminus\Omega_k^+}=v|_{\Omega\setminus\Omega_k^+}$.
  On the other hand, we can
  conclude $\Pi_kv|_{\Omega_k^{++}}=v|_{\Omega_k^{++}}$
  from~\eqref{Pik:4}. This yields~\eqref{Pik:4a}.

  The claim \eqref{Pik:5a} is
  an immediate consequence of \eqref{Pik:3}.
\end{proof}

\begin{lem}[Stability]\label{lem:PikStab}
    Let $v\in\Vk[\ell]$ for some $ k\le
    \ell\in\N_0\cup\{\infty\}$. Then for all $\elm\in\gridk$, we have 
    \begin{multline*}
    \int_\elm\abs{\nabla\Pi_kv}^2\dx+\int_{\partial\elm}h_k^{-1}\jump{\Pi_kv}^2\ds
    \\\Cleq
    \int_{\omegak(\elm)}\abs{\nablaG v}^2\dx+\sum_{\elm'\in\gridk[\ell], \elm'\subset
      \omegak(\elm)}
    \int_{\partial\elm'}h_\ell^{-1}\jump{v}^2\ds,
  \end{multline*}
 setting $\grid_\ell:=\grid^+$ and $h_\ell:=h_+$, when $\ell=\infty$.
 In particular, we  have 
    $\enorm[k]{\Pi_k v} \Cleq
    \enorm[\ell]{v}$.
  %
  % The  hidden constants depend only on $d$, $r$ and the shape
  %   regularity of $\gridk[0]$. 
\end{lem}
\begin{proof}
  We begin by noting that, summing over all elements in $\gridk$ and
  accounting for the finite overlap of the domains $\omegak(\elm)$,
  $\elm\in\gridk$,  the global stability estimate is an immediate
  consequence of the corresponding local one.

  We first assume $\ell<\infty$. Let
  $\elm\in\gridk^{++}\subset\gridk[\ell]^{++}$. Then, thanks to
  Lemma~\ref{lem:Pik}\eqref{Pik:4}, we have  
  $\Pi_kv|_\elm=v|_\elm$. Moreover, let $\elm'\in\gridk$ such that
  $\elm\cap\elm'\in\sides_k$; then $\neighk(z)\ni\elm\in \gridk^{++}$
  and thus $(\Pi_kv)|_{\elm'}(z)=v|_{\elm'}(z)$, for all
  $z\in\nodes_k(\elm)\cap\nodes_k(\elm')$.  Consequently, we have
  $\jump{\Pi_kv}=\jump{v}$
  on $\partial\elm$, in other words
  \begin{align}
    \int_\elm\abs{\nabla\Pi_kv}^2\dx+\int_{\partial\elm}h_k^{-1}\jump{\Pi_kv}^2\ds 
    = \int_\elm\abs{\nabla v}^2\dx+\int_{\partial\elm}h_k^{-1}\jump{v}^2\ds .
  \end{align}

  Let now $\elm\in\grid_k$ be arbitrary. Then, an inverse estimate and the local stability
  (Lemma~\ref{lem:Pik} \eqref{Pik:1} and~\eqref{Pik:5}) for
  $v_\elm:=\frac1{|\omegak(\elm)|}\int_{\omegak(\elm)} v\dx\in\R$, imply 
  \begin{align}\label{eq:localstab}
    \begin{aligned}
      \int_\elm\abs{\nabla\Pi_kv}^2\dx&\Cleq
      \int_\elm h_k^{-2}\abs{\Pi_k(v-v_\elm)}^2\dx\Cleq
      \int_{\omegak(\elm)}h_k^{-2}\abs{v-v_\elm}^2\dx
      \\
      &\Cleq
      \sum_{\elm'\subset
        \omegak(\elm),\elm'\in\gridk[\ell]}\int_{\elm'}\abs{\nabla v}^2\dx+
      \int_{\partial\elm'}h_{\ell}^{-1}\jump{v}^2\ds;
    \end{aligned}
  \end{align}
  here the last estimate follows from the broken Poincar\'{e} inequality,
  Proposition~\ref{prop:poincareG}. 

  If now for all $\elm'\in\gridk$, with $\elm'\subset\omegak(\elm)$, we have
  $\elm'\not\in\gridk^{++}$, which
  implies $\elm\in\gridk\setminus\gridk^{++}$. Then,
  thanks to Lemma~\ref{lem:Pik}\eqref{Pik:3}, we
  have that $\Pi_kv$ is continuous across $\partial\elm$, i.e.,
  $\jump{\Pi_kv}|_{\partial\elm}=0$. 
  On the contrary, assuming that there exists $\elm'\in\gridk^{++}$,
  with $\elm'\in\neighk(\elm)$, we conclude that
  $\elm\in\neighk(\elm')$ and thus $\elm\in\grid^+$.
  From the local quasi
  uniformity, we thus have for all $\elm''\in\grid_\ell$ with  $\elm''\cap\elm\neq\emptyset$
  that $|\elm''|\eqsim |\elm|$.
  Let $z\in \nodes_k(\elm)$; then, according to \eqref{df:Pikz}, we have that
  \begin{align*}
    \jump{\Pi_kv}|_{\partial\elm}(z)=
    \begin{cases}
      \jump{v}|_{\partial\elm}(z),\quad&\text{if}~\exists 
      \elm'\in\neighk(z)\cap\gridk^{++};
      \\
      0,&\text{else.}
    \end{cases}
  \end{align*}
  Using standard scaling arguments, this implies
  \begin{align*}
    \int_{\partial\elm}\jump{\Pi_kv}^2\ds&\eqsim
    |\partial\elm|\sum_{z\in\nodes_k\cap\partial\elm}
    \big(\jump{\Pi_kv}|_{\partial\elm}(z)\big)^2
                                           = 
    |\partial\elm|\sum_{z\in\nodes_k\cap\partial\elm}
    \big(\jump{v}|_{\partial\elm}(z)\big)^2
    \\
    &\le
     |\partial\elm|\sum_{z\in\nodes_\ell\cap\partial\elm}\big(\jump{v}|_{\partial\elm}(z)\big)^2
      \eqsim \int_{\partial\elm}\jump{v}^2\ds.
  \end{align*}
  Combining this with~\eqref{eq:localstab} proves the local
 bound in the case $\ell<\infty$.

  For $\ell=\infty$, we observe
  that a bound similar to~\eqref{eq:localstab} can be obtained with
  Lemma~\ref{lem:poincare} instead of
  Proposition~\ref{prop:poincareG}. The local bound follows then by
  arguing as in the case $\ell<\infty$.
\end{proof}

\begin{cor}[Interpolation estimate]\label{cor:PikIpol}
  For $v\in\V_\ell$, $k\le\ell\in\N\cup\{\infty\}$, we have that 
  \begin{multline*}
    \int_\elm|\nablaG v-\nablaG \Pi_k v|^2\dx + \int_\elm \hk^{-2}
    |v-\Pi_k v|^2+\int_{\partial\elm}\hk^{-1}
    \jump{v-\Pi_k v}^2
    \\\Cleq \int_{\omegak(\elm)}|\nablaG v|^2\dx +\sum_{S\in\sides_\ell,S\subset\omegak(\elm)}\int_{S}\hk^{-1}
    \jump{v}^2,
  \end{multline*}
  where we set $\grid_\ell:=\grid^+$ and $h_\ell:=h_+$, when $\ell=\infty$.
  The constant depends only on $d$, $r$ and the shape
    regularity of $\gridk[0]$.
\end{cor}
\begin{proof}
  The claim follows from Lemma~\ref{lem:Pik}\eqref{Pik:5}, together
  with the stability Lemma~\ref{lem:PikStab} and the local Poincar\'e inequality
  from Proposition~\ref{prop:poincareG}, respectively,
  Lemma~\ref{lem:poincare}. 
\end{proof}

The next result concerns the convergence of the quasi-interpolation.
\begin{lem}\label{lem:Pik2}
  Let $v\in\V_\infty$; then,
  \begin{align*}
    \enorm[k]{v-\Pi_k v}\to 0 \qquad\text{and} \qquad \enorm[\infty]{v-\Pi_k v}\to 0
  \end{align*}
  as $k\to\infty$.
\end{lem}

\begin{proof}
 For brevity, set
  $v_k:=\Pi_k v\in\Vk$. Thanks to Lemma~\ref{lem:poincare} and Lemma~\ref{lem:Pik}\eqref{Pik:4} and
  \eqref{Pik:3},
  we have that
  \begin{align*}
    \enorm[k]{v-v_k}^2&\Cleq\int_{\gridk\setminus\gridk^{++}}|\nablaG v-\nablaG
    v_k|^2\dx+
    \int_{\sides_k\setminus\sides_k^{++}}\hG[k]^{-1}\abs{\jump{v-v_k}}^2\ds
    \\
    &\le\int_{\gridk^-}|\nablaG v-\nablaG
    v_k|^2\dx+
    \int_{\sides_k^-}\hG[k]^{-1}\abs{\jump{v-v_k}}^2\ds
      \\
    &=I^-_k +II^-_k.
  \end{align*}
  We conclude from  Lemma~\ref{lem:PikStab} that
  \begin{align*}
   II^-_k&=
               \int_{\sides_k^-}\hG[k]^{-1}\abs{\jump{v-v_k}}^2\ds
    \Cleq \sum_{\elm\in\gridk^-}\sum_{\elm'\in\grid^+,
    \elm'\subset\omegak(\elm)}~\int_{\partial\elm'} \hG[+]^{-1}\jump{v}^2\ds
    \\
    &\Cleq \int_{\sides^+\setminus\sides_k^{++}} \hG[+]^{-1}\jump{v}^2\ds.
  \end{align*}
  The term on the right hand side is the tail of a convergent series, since it is
  bounded thanks to 
  $\enorm[\infty]{v}<\infty$ and all of its summands are positive. Therefore,  $II_k^-\to0$ as
  $k\to\infty$. 

  Thus, it remains to prove that $I_k^-\to 0$ as $k\to \infty$. To this
  end, we recall that thanks to the  definition of
  $H^1_{\partial\Omega\cap\partial\Omega^-}(\Omega^-)$ we have that
  $v|_{\Omega^-}=\tilde v|_{\Omega^-}$ for some function $\tilde v\in H_0^1(\Omega)$. Since $H_0^2(\Omega)$ is dense in $H_0^1(\Omega)$, for $\epsilon>0$, there exists $v_\epsilon\in H^2_0(\Omega)$ 
  such
  that $\|\tilde v-v_\epsilon\|_{H^1(\Omega^-)}\le
  \|\tilde v-v_\epsilon\|_{H^1(\Omega)}<\epsilon$.
  Combining
  Lemma~\ref{lem:Pik}\eqref{Pik:5} and \eqref{Pik:1} with
  standard estimates
  \cite{ScottZhang:90,DemlowGeorgoulis:12} for $H_0^1(\Omega)$
  functions, with
  the  Bramble-Hilbert Lemma (see, e.g., \cite{BrennerScott:02}),
  we 
  obtain 
  \begin{multline*}
    \int_{\gridk^-}|\nablaG v-\nabla
    v_k|^2\dx\\
    \begin{aligned}
      &\Cleq \int_{\gridk^-}|\nabla v_\epsilon-\nabla
      \Pi_kv_\epsilon|^2+|\nablaG (v- v_\epsilon)-\nabla
      \Pi_k(v-v_\epsilon)|^2\dx
      \\
      &\Cleq
      \int_{\neighk(\gridk^-)}\hG[k]^2\sum_{|\alpha|=2}|D^\alpha
      v_\epsilon|^2\dx+\int_{\neighk(\gridk^-)}|\nablaG
      (v-v_\epsilon)|^2\dx
      \\
      &\Cleq \|h_k\chi_{\Omega_k^-}\|_{L^\infty(\Omega)}^2
      \int_{\Omega}\sum_{|\alpha|=2}|D^\alpha
      v_\epsilon|^2\dx+\int_{\neighk(\gridk^-)}|\nablaG
      (v-v_\epsilon)|^2\dx,
    \end{aligned}
  \end{multline*}
  where we have used that
  $\|h_k\|_{L^\infty(\Omega(\neighk(\gridk^-)))}\Cleq
  \|h_k\chi_{\Omega_k^-}\|_{L^\infty(\Omega)}\to 0$ as $k\to\infty$, thanks to the local quasi-uniformity
  of $\gridk$ and 
  Lemma~\ref{lem:Omegastar}. Moreover, we conclude
  $\int_{\neighk(\gridk^-)}|\nablaG (v-v_\epsilon)|^2\dx\rightarrow
  \int_{\Omega^-}|\nabla (v-v_\epsilon)|^2\dx<\epsilon$ from
  Lemma~\ref{lem:Omegastar} and the absolute continuity of the
  Lebesgue integral. Consequently, 
  $\lim_{k\to\infty}I^-_k\Cleq \epsilon $, which completes the proof
  of the first claim, since $\epsilon>0$ is arbitrary. 

  The second claim follows similarly by replacing $\sides_k$ by
  $\sides^+$ and noting that $\enorm[k]{\Pi_kv}=\enorm[\infty]{\Pi_k v}$,
  since $\Pi_k v$ is continuous in $\Omega\setminus\Omega^+$.
\end{proof}

\begin{proof}[Proof of Lemma~\ref{lem:V}]
  The positivity of $\enorm[\infty]{\cdot}$ on $\V_\infty$
    follows from Lemma~\ref{lem:Pik2} together with 
    Corollary~\ref{C:FriedrichsVinfty}.  % and Proposition~\ref{P:|Dv|<dG}.

  In order to prove that $\V_\infty$ is complete with respect to
  $\enorm[\infty]{\cdot}$, let 
  $\{v^\ell\}_{\ell\in\N}\subset\V_\infty$ be a Cauchy sequence with
  respect to $\enorm[\infty]{\cdot}$. Note that thanks to the Friedrichs
  inequality (Corollary~\ref{C:FriedrichsVinfty}),
  there exists the limit $v^\ell\to v\in L^2(\Omega)$; this is the
  candidate for the limit of $v^\ell$ in $\V_\infty$. 

  We first observe that, 
  since 
  $v^\ell|_\elm\in\P_r$ for all $\elm\in\grid^+$, it follows
  from the definition of $\enorm[\infty]{\cdot}$ that 
  $v|_\elm\in\P_r$ for all $\elm\in\grid^+$.
  Moreover, 
  Propositions~\ref{P:|Dv|<dG} and~\ref{lem:Pik2} imply
  that $\{v^\ell\}_{\ell\in\N}$ is also a Cauchy Sequence in
  $BV(\Omega)$ and thus $v\in BV(\Omega)$.
  Therefore, $v$ has $L^1$-traces on each $\partial \elm$,
  \(\elm\in\gridk\), $k\in
  \N$.

  Next, we deal with the jump terms. To this end, we first
  observe that, for $k\in\N$, $\{v^\ell\}_{\ell\in\N}$ is also a Cauchy sequence with
  respect to the $\enorm[k]{\cdot}$-norm and uniqueness of
  limits imply  $\jump{v^\ell}|_\side\to \jump{v}|_\side$ in $L^2(\side)$ as
  $\ell\to\infty$, for all
  $\side\in\sides_k$, in the sense of traces.
  Let $\epsilon>0$ arbitrary fixed, then there exists $L=L(\epsilon)$,
  such that $\enorm[k]{v^j-v^\ell}\le \enorm[\infty]{v^j-v^\ell}\le
  \epsilon$ for all $j\ge \ell\ge L$. Fix $\ell\ge L$, then thanks to
  Proposition~\ref{prop:V}, there exists $K\equiv K(\epsilon,L)$, such that for all $k\ge K$, we have
  \begin{align*}
    \int_{\sides_k\setminus\sides_k^+}h_+^{-1} \jump{v^L}^2\ds\le
  \epsilon^2.
  \end{align*}
  Consequently, we have
\begin{align*}
  \begin{aligned}
    \int_{\sides_k} \hG[k]^{-1} \jump{v}^2 \ds &=
    \int_{\sides_k\setminus\sides_k^+} \hG[k]^{-1} \jump{v}^2
    \ds+\int_{\sides_k^+} \hG[k]^{-1} \jump{v}^2 \ds
    \\
    &=\lim_{\ell\to\infty}\int_{\sides_k\setminus\sides_k^+} \hG[k]^{-1}
    \jump{v^\ell}^2 \ds+\int_{\sides_k^+} \hG[k]^{-1} \jump{v}^2
    \ds.
  \end{aligned}
\end{align*}
Thus,  it follows from
\begin{align}
  \int_{\sides_k\setminus\sides_k^+} \hG[k]^{-1}
\jump{v^\ell}^2 \ds\le 2\enorm[k]{v^\ell-v^L}^2+2\int_{\sides_k\setminus\sides_k^+} \hG[k]^{-1}
\jump{v^L}^2 \ds\le 4\epsilon^2,\label{eq:bnd-restjump}
\end{align}
for $\ell\ge L$, that
\begin{align}\label{conv:jumps_CS}
  \int_{\sides_k} \hG[k]^{-1}
  \jump{v}^2 \ds \to
  \int_{\sides^+} \hG[+]^{-1}
  \jump{v}^2 \ds \qquad\text{as}~ k\to\infty,
\end{align}
since $\epsilon>0$ is arbitrary.

We next verify that $v|_{\Omega^-}\in
    H^1_{\partial\Omega\cap\partial\Omega^-}(\Omega^-)$, i.e., that $v$ is a
  restriction of a function from $H_0^1(\Omega)$ to $\Omega^-$. To
  this end, for each $\ell,m\in\N$, we
  define $v^\ell_m:=\Pi_mv^\ell\in\V_m$ for $\ell\ge m\in\N$ and since $v^\ell_m\in
  C(\Omega\setminus\Omega^{+}_k)\subset
  C(\Omega\setminus\Omega^{+}_m)$ (see
  Lemma~\ref{lem:Pik}\eqref{Pik:3}) for $k\ge m$,  we have that
  $\enorm[m]{v_m^\ell}=\enorm[k]{v_m^\ell}=\enorm[\infty]{v_m^\ell}$. % Therefore,
  % \begin{align*}
  %   \enorm[\infty]{v-v_m^\ell}\le     \enorm[\infty]{v-v^\ell}+    \enorm[\infty]{v^\ell-v_m^\ell}. 
  % \end{align*}
  Thanks to Lemma~\ref{lem:Pik2}, for each $\ell\in\N$, there exists a
  monotone sequence 
  $\{m_\ell\}_\ell\in\N$, such that $\enorm[\infty]{v^\ell-v^\ell_{m_\ell}}\le \frac1\ell$ and
  thus
   \begin{align*}
    \enorm[m_\ell]{v_{m_\ell}^\ell}=\enorm[\infty]{v_{m_\ell}^\ell}\le \enorm[\infty]{v^\ell-v_{m_\ell}^\ell}
+\enorm[\infty]{v^\ell}<\frac1\ell+\enorm[\infty]{v^\ell}<\infty 
   \end{align*}
   Consequently, the conforming interpolation
   $\mathcal{I}_{m_\ell}v_{m_\ell}^\ell\in\V_{m_\ell}\cap H_0^1(\Omega)$ 
  from Proposition~\ref{P:dist-dGcG} is bounded uniformly in
  $H_0^1(\Omega)$ and thus there exists a weak limit $\tilde v\in
  H_0^1(\Omega)$ of a subsequence, which for convenience we denote with the same label. Moreover, again from
  Proposition~\ref{P:dist-dGcG}, we have
  \begin{align}\label{eq:L2v-Imvm}
    \begin{aligned}
      \norm[\Omega^-]{v-\mathcal{I}_{m_\ell}v^\ell_{m_\ell}}&\le \norm[\Omega]{v-v^\ell}+
      \norm[\Omega]{v^\ell-v^\ell_{m_\ell}}+
      \norm[\Omega^-]{v^\ell_{m_\ell}-\mathcal{I}_{m_\ell}v^\ell_{m_\ell}}
      \\
      &\Cleq \norm[\Omega]{v-v^\ell}+\frac1\ell+
      \int_{\sides(\gridk[m_\ell]\setminus\gridk[m_\ell]^+)}\jump{v^\ell_{m_\ell}}^2\ds
      \\
      &\le \norm[\Omega]{v-v^\ell}+\frac2\ell+ \norm[L^\infty(\Omega)]{\chi_{\Omega_{m_\ell}^-}h_{m_\ell}}\enorm[m_\ell]{v^\ell}^2,
    \end{aligned}
  \end{align}
  which vanishes as $\ell\to\infty$ thanks to the Friedrichs
  inequality (Corollary~\ref{C:FriedrichsVinfty}) and
  Lemma~\ref{lem:Omegastar}. Therefore, $v|_{\Omega^-}=\tilde
  v|_{\Omega^-}$ and we can define the piecewise gradient of $v$ as
  in~\eqref{df:nablapw}.

  We shall next show that $\enorm[\infty]{v-v^\ell}\to0$ as $\ell \to
  \infty$. Arguing similar as for~\eqref{conv:jumps_CS}, we have 
  \begin{align*}
    \int_{\sides^+}h_+^{-1}\jump{v-v^\ell}^2\ds\to
    0\quad\text{as}~\ell\to \infty.
  \end{align*}
  Consequently, it remains to show that $\norm{\nablaG v -\nablaG
    v^\ell}\to 0$ as $\ell\to\infty$. To this end, we observe that
  $\{\nablaG v^\ell\}_\ell$ is a Cauchy Sequence in $L^2(\Omega)^d$
  and thus there exists $\vec d\in L^2(\Omega)^d$ with $\norm{\nablaG
    v^\ell-\vec d}\to 0$ as $\ell\to\infty$ and it thus suffices to prove
  $\vec d=\nablaG v$. Let $\phi\in C_0^\infty(\Omega)$, then we have
  from Lemma~\ref{lem:Pik2}
  for the distributional derivative on the one hand, that
  \begin{align*}
    \langle Dv^\ell,\phi\rangle&=\int_{\Omega}
      \nablaG v^\ell \cdot\phi \dx -\int_{\sides^+} \jump{v^\ell}\cdot\phi\ds
                                 \to \int_{\Omega}
      \vec d \cdot\phi \dx -\int_{\sides^+} \jump{v}\cdot\phi\ds
  \end{align*}
  as $\ell\to\infty$.
  On the other hand, 
  \begin{align*}
    \begin{aligned}
      \langle Dv^\ell,\phi\rangle&=\int_{\Omega\setminus\Omega_k^+}
      \nablaG v^\ell \cdot\phi \dx +\int_{\Omega_k^+} \nablaG v^\ell
      \cdot\phi \dx -\int_{\sides^+} \jump{v^\ell}\cdot\phi\ds
      \\
      &=\int_{\Omega\setminus\Omega_k^+}
      \nabla\mathcal{I}_{m_\ell}v^\ell_{m_\ell} \cdot\phi \dx
      +\int_{\Omega\setminus\Omega_k^+} \nablaG
      (v^\ell-\mathcal{I}_{m_\ell}v^\ell_{m_\ell})
      \cdot\phi \dx\\
      &\quad+\int_{\Omega_k^+} \nablaG v^\ell \cdot\phi \dx
      -\int_{\sides^+} \jump{v^\ell}\cdot\phi\ds.
    \end{aligned}
  \end{align*}
  In order to estimate the second term, we employ
  Proposition~\ref{P:dist-dGcG}, and obtain for some arbitrary given
  $\epsilon>0$ that 
  \begin{align*}
    \int_{\Omega\setminus\Omega_k^+}\abs{ \nablaG
    (v^\ell-\mathcal{I}_{m_\ell}v^\ell_{m_\ell})}^2\dx&\Cleq \int_{\Omega}\abs{ \nablaG
                                                        (v^\ell-v^\ell_{m_\ell})}^2\dx
    \\
                                                      &\quad+\int_{\Omega\setminus\Omega_k^+}\abs{ \nablaG
                                                        (v^\ell_{m_\ell}-\mathcal{I}_{m_\ell}v^\ell_{m_\ell})}^2\dx
    \\
                                                      &\Cleq \frac1\ell+\int_{\sides_k^-}h_{m_\ell}^{-1}\jump{v^\ell}^2\ds\le\frac1\ell+\epsilon^2
  \end{align*}
  for all $\ell\ge L(\epsilon)$ and $k\ge K(\epsilon,L)$ similarly as in~\eqref{eq:bnd-restjump}.
  Recalling that $\tilde v$ is the weak limit of
  $\{\mathcal{I}_{m_\ell}v^\ell_{m_\ell}\}_\ell$ in $H_0^1(\Omega)$
  and $v^\ell$ converges strongly in $\P_r(\grid_k^+)$,
  we thus conclude that
  \begin{align*}
    \abs{\int_\Omega\big(\chi_{\Omega\setminus\Omega_k^+}\nabla\tilde
    v+\chi_{\Omega_k^+}\nablaG v-\vec d\big)\cdot\phi\dx}\le \epsilon \norm[\Omega]{\phi}.
  \end{align*}
  Recalling~\eqref{df:nablapw} the assertion follows by letting $k\to\infty$ from the uniform integrability
  of $\nabla \tilde v$ and $\nablaG v|_{\Omega^+}$.

  Finally note that $\V_{k}\subset\V_j$ for $j\ge k$ and thus defining 
  $w_k:=v_{m_\ell}^\ell$ for  $k\in\{m_\ell,\ldots,m_{\ell+1}-1\}$
  yields $w_k\in\V_k$. Consequently, 
  \begin{align*}
  \norm{v-w_k}+\enorm[k]{v-w_k}\Cleq
  \enorm[\infty]{v-w_k}=\enorm[\infty]{v-v_{m_\ell}^\ell}\le
  \enorm[\infty]{v-v^\ell}+\frac1\ell,
  \end{align*}
  where we have used that the
  Friedrichs inequality (Corollary~\ref{C:FriedrichsVinfty}) is
  inherited since
  \(\enorm[\infty]{v-w_k}=\lim_{\ell\to\infty}\enorm[\infty]{v_\ell-w_k}\).
The right-hand side
  vanishes because $\ell\to
  \infty$ as
  $k\to\infty$. 
\end{proof}

\begin{proof}[Proof of Corollary~\ref{cor:u8}]
  The assertion follows from Lemma~\ref{lem:V} and the observation that 
  \begin{align*}
    \enorm[\infty]{v}^2\Cleq\bilin[\infty]{v}{v}\qquad\text{and}
    \qquad\bilin[\infty]{v}{w}\Cleq \enorm[\infty]{v}\enorm[\infty]{w}
  \end{align*}
  for all $v,w\in\V_\infty$. Indeed, the continuity follows with standard
  techniques using~\eqref{eq:lift-stab} and the coercivity is a consequence of 
  \begin{align*}
    % \bilin[\infty]{v}{\Pi_kw}\Cleq \enorm[k]{\Pi_kv}\enorm[k]{\Pi_kw}=\enorm[\infty]{\Pi_kv}\enorm[\infty]{\Pi_kw}
    % \intertext{and}
    \enorm[\infty]{\Pi_kv}^2=
    \enorm[k]{\Pi_kv}^2\Cleq \bilin[k]{\Pi_kv}{\Pi_kv}=\bilin[\infty]{\Pi_kv}{\Pi_kv}
  \end{align*}
  and Lemma~\ref{lem:Pik2}.
\end{proof}

\section{(Almost) best approximation property}
\label{sec:uinfty}

In this section we shall prove that the solution
$u_\infty\in\V_\infty$ of~\eqref{eq:u8} is indeed
the limit of the discontinuous Galerkin solutions produced by
\ADGM. 
This is a consequence of the density of spaces
$\{\V_k\}_{k\in\N_0}$ in $\V_\infty$ and the (almost) best approximation
property of discontinuous Galerkin solutions; the latter generalises \cite{Gudi:10}.

\begin{lem}\label{lem:conv1}
  Let $u_\infty\in\V_\infty$ be the solution of~\eqref{eq:u8} and $u_k\in\Vk$ be the
  \DGFEM approximation from \eqref{ipdg} on $\gridk$ for some $k\in\N$
  and $u_\infty$ the unique solution of the limit problem from
  Corollary~\ref{cor:u8}.
  Then, we have
  \begin{align*}
    \enorm[k]{u_\infty-u_k}\Cleq \enorm[\infty]{u_\infty-\Pi_ku_\infty}+\frac{\scp{f}{u_k-\Pi_ku_k}-\bilin[k]{\Pi_ku_\infty}{u_k-\Pi_ku_k}}{\enorm[k]{u_k-\Pi_ku_\infty}}.
  \end{align*}
%  where $\psi=u_k-\Pi_ku_\infty$.
\end{lem}
\begin{proof}
  Assume that $u_k\neq\Pi_ku_\infty\in \Vk\cap\V_\infty$ and set $\psi=u_k-\Pi_ku_\infty$. Then, we have from~\eqref{eq:coercive} that
  \begin{align*}
    \alpha\enorm[k]{u_k-\Pi_ku_\infty}^2&\le\bilin[k]{u_k-\Pi_ku_\infty}{\psi}=\scp{f}{\psi}-\bilin[k]{\Pi_ku_\infty}{\psi}
    \\
    &=\scp{f}{\Pi_k\psi}+\scp{f}{\psi-\Pi_k\psi}-\bilin[k]{\Pi_ku_\infty}{\psi}
      \\
    &=\big(\bilin[\infty]{u_\infty}{\Pi_k\psi}-\bilin[k]{\Pi_ku_\infty}{\Pi_k\psi}\big)
      \\
    &\quad
      +\big(\scp{f}{\psi-\Pi_k\psi}-\bilin[k]{\Pi_ku_\infty}{\psi-\Pi_k\psi}\big)\equiv (I)+(II),
  \end{align*}
 using that $\Pi_k\psi\in \V_k\cap\V_\infty$ from Lemma~\ref{lem:Pik}\eqref{Pik:5a}. 
% Therefore, we obtain
%  \begin{align*}
%    \alpha\enorm[k]{u_k-\Pi_ku_\infty}&\le
%                            \frac{\bilin[\infty]{u_\infty}{\Pi_k\psi}-\bilin[k]{\Pi_ku_\infty}{\Pi_k\psi}}{\enorm[k]{u_k-\Pi_ku_\infty}}
%                            \\
%    &\quad+\frac{\scp{f}{\psi-\Pi_k\psi}-\bilin[k]{\Pi_ku_\infty}{\psi-\Pi_k\psi}}{\enorm[k]{u_k-\Pi_ku_\infty}}.
%  \end{align*}
  For $(I)$, we have, respectively, 
  \begin{align*}
   % \bilin[\infty]{u_\infty}{\Pi_k\psi}-\bilin[k]{\Pi_ku_\infty}{\Pi_k\psi}
   % \\
   % \begin{split}
     (I) &=\int_{\Omega}\nablaG u_\infty\cdot\nablaG \Pi_k\psi\,\ud x
      -\int_{\sides^+}\big(\mean{\nabla
        u_\infty}\cdot\jump{\Pi_k\psi}+\theta \mean{\nabla
        \Pi_k\psi}\cdot\jump{u_\infty}\big)\ds
      \\
      &\quad
      +\int_{\mathring\sides^+}
       \big(\vec{\beta}\cdot\jump{u_\infty}\jump{\nabla \Pi_k\psi}+\jump{\nabla
         u_\infty}\vec{\beta}\cdot\jump{\Pi_k\psi}\big)\ds
       \\
       &\quad + \int_\Omega \gamma
         \big(\riftG[\infty](\jump{u_\infty})+\liftG[\infty](\vec{\beta}\cdot\jump{u_\infty})\big)\cdot
         \big(\riftG[\infty](\jump{\Pi_k\psi})+\liftG[\infty](\vec{\beta}\cdot\jump{\Pi_k\psi})\big)\dx 
      \\
      &\quad+\int_{\sides^+}\frac{\sigma}{\hG[+]}\jump{u_\infty}\cdot\jump{\Pi_k\psi}\,\ud
      s
      \\
      &\quad-\int_{\Omega}\nablaG \Pi_ku_\infty\cdot\nablaG \Pi_k\psi\,\ud x
      +\int_{\sides_k}\big(\mean{\nabla
        \Pi_ku_\infty}\cdot\jump{\Pi_k\psi}+\theta \mean{\nabla
        \Pi_k\psi}\cdot\jump{\Pi_ku_\infty}\big)\ds
      \\
       &\quad -\int_{\mathring\sides^+}
       \big(\vec{\beta}\cdot\jump{\Pi_ku_\infty}\jump{\nabla \Pi_k\psi}+\jump{\nabla
         \Pi_ku_\infty}\vec{\beta}\cdot\jump{\Pi_k\psi}\big)\ds
       \\
       &\quad - \int_\Omega \gamma
         \big(\riftG[k](\jump{\Pi_ku_\infty})+\liftG[k](\vec{\beta}\cdot\jump{\Pi_ku_\infty})\big)\cdot
         \big(\riftG[k](\jump{\Pi_k\psi})+\liftG[k](\vec{\beta}\cdot\jump{\Pi_k\psi})\big)\dx 
      \\
      &\quad-\int_{\sides^+}\frac{\sigma}{\hG[k]}\jump{\Pi_ku_\infty}\cdot\jump{\Pi_k\psi}\,\ud
      s
      \\
      % &=\int_{\Omega}\nablaG (u_\infty-\Pi_ku_\infty)\cdot\nablaG \Pi_k\psi\,\ud x
      % \\
      % &\quad-\int_{\sides^+}\mean{\nabla
      %   \Pi_k\psi}\cdot\jump{u_\infty-\Pi_ku_\infty}\ds
      % +\int_{\sides^+_k}\mean{\nabla
      %   (u_\infty-\Pi_ku_\infty)}\cdot\jump{\Pi_k\psi}
      % \ds
      % \\
      % &\quad+\int_{\sides^+_k}\big(\frac{\sigma}{\hG[+]}\jump{u_\infty}
      % -\frac{\sigma}{\hG[k]}\jump{\Pi_ku_\infty}\big)\cdot\jump{\Pi_k\psi}\,\ud
      % s 
      % \\
      &=\int_{\Omega}\nablaG (u_\infty-\Pi_ku_\infty)\cdot\nablaG \Pi_k\psi\,\ud x
      \\
      &\quad -\int_{\sides^+_k}\mean{\nabla
        (u_\infty-\Pi_ku_\infty)}\cdot\jump{\Pi_k\psi}
      \ds-\theta\int_{\sides^+}\mean{\nabla
        \Pi_k\psi}\cdot\jump{u_\infty-\Pi_ku_\infty}\ds
       \\
      &\quad +\int_{\mathring\sides^+}
       \big(\vec{\beta}\cdot\jump{u_\infty-\Pi_ku_\infty}\jump{\nabla
         \Pi_k\psi}+\jump{\nabla u_\infty-\nabla
         \Pi_ku_\infty}\vec{\beta}\cdot\jump{\Pi_k\psi}\big)\ds
      \\
      &\quad+ \int_\Omega \gamma
        \big(\riftG[\infty](\jump{u_\infty-\Pi_ku_\infty})+\liftG[\infty](\vec{\beta}\cdot\jump{u_\infty-\Pi_ku_\infty})\big)
        \\
    &\qquad\cdot \big(\riftG[\infty](\jump{\Pi_k\psi})+\liftG[\infty](\vec{\beta}\cdot\jump{\Pi_k\psi})\big)\dx
    \\
      &\quad+\int_{\sides^+_k}\frac{\sigma}{\hG[k]}\jump{u_\infty-\Pi_ku_\infty}
      \cdot\jump{\Pi_k\psi}\,\ud
      s
      \\
      &\Cleq \enorm[\infty]{u_\infty
        -\Pi_ku_\infty}\enorm[\infty]{\Pi_k\psi}=\enorm[\infty]{u_\infty
        -\Pi_ku_\infty}\enorm[k]{\Pi_k\psi}
      \\
      &\Cleq \enorm[\infty]{u_\infty
        -\Pi_ku_\infty}\enorm[k]{u_k-\Pi_ku_\infty};
   % \end{split}
  \end{align*}
  here we used that $\Pi_ku_\infty,\Pi_k\psi\in\Vk\cap\V_\infty$, $h_\infty=h_k$
  on $\sides_k^+$ and that $\Pi_ku_\infty$
  and $\Pi_k\psi$ are continuous on $\Omega\setminus\Omega^+_k$, i.e.,
  $\jump{\Pi_ku_\infty}
  =\jump{\Pi_k\psi}=0$ on $\sides^+\setminus\sides_k^+$, 
  which follows from
  Lemma~\ref{lem:Pik}. 
  Note that this and $\jump{\Pi_ku_\infty}
  =\jump{\Pi_k\psi}=0$ on 
  $\partial(\Omega\setminus\Omega_k^+)$ from Lemma~\ref{lem:Pik}
  also implies that
  $\liftG[k](\Pi_k\psi)=\liftG[\infty](\Pi_k\psi)$ and
  $\liftG[k](\Pi_ku_\infty)=\liftG[\infty](\Pi_ku_\infty)$ as well as
  the corresponding relations between $\riftG[k]$ and $\riftG[\infty]$;
  compare with~\eqref{eq:lift-inf}. Thus, the above estimate follows
  from the Cauchy-Schwarz inequality, application of inverse inequalities in
  conjunction with the stability of the lifting
  operators~\eqref{eq:lift-stab},
  and Lemma~\ref{lem:PikStab}.

  Consequently, triangle
  inequality and the above imply
  \begin{align*}
    \enorm[k]{u_\infty-\uk}&\leq
    \enorm[k]{u_\infty-\Pi_ku_\infty}+\enorm[k]{\uk-\Pi_ku_\infty}
    \\
    &\Cleq
    \enorm[k]{u_\infty-\Pi_ku_\infty}+\enorm[\infty]{u_\infty-\Pi_ku_\infty}
      \\
    &\qquad+
    \frac{\scp{f}{\psi-\Pi_k\psi}-\bilin[k]{\Pi_ku_\infty}{\psi-\Pi_k\psi}}{\enorm[k]{u_k-\Pi_ku_\infty}}. 
  \end{align*}
  Thanks to $\enorm[k]{u_\infty-\Pi_ku_\infty}\le\enorm[\infty]{u_\infty-\Pi_ku_\infty}$,
  this proves the assertion.
\end{proof}

The properties of  the quasi-interpolation~\eqref{df:Pik} allow for the
consistency term in Lemma~\ref{lem:conv1} to be
bounded
by the a posteriori indicators of essentially the elements, which 
will experience further refinements. 

\begin{lem}\label{lem:conv2}
  Let $u_\infty\in\V_\infty$ be the solution of~\eqref{eq:u8} and $u_k\in\Vk$ be the
  \DGFEM approximation from \eqref{ipdg} on $\gridk$ for some $k\in\N$. Then, we have 
  \begin{align*}
    \frac{\scp{f}{u_k-\Pi_ku_k}-\bilin[k]{\Pi_ku_\infty}{u_k-\Pi_ku_k}}{\enorm[k]{u_k-\Pi_ku_\infty}}
    \Cleq 
    \Big(\sum_{\elm\in\grid_k\setminus\grid_k^{3+}}\est_k(\Pi_ku_\infty,\elm)^2\Big)^{1/2},
  \end{align*}
  where $\gridk^{3+}:=\{\elm\in\gridk:\neighk(\elm)\subset \gridk^{++}\}$.
\end{lem}

\begin{proof}
  Let $v_k:=\Pi_ku_\infty$ and $\phi:=u_k-\Pi_k
  u_k=u_k-\Pi_ku_\infty-\Pi_k(u_k-\Pi_ku_\infty)$. Then, using
  integration by parts, we have
  \begin{align*}
    \scp{f}{\phi}&-\bilin[k]{v_k}{\phi}
    \\
    &=\int_{\gridk}(f+\Delta
                                        v_k)\phi\dx
                                        -\int_{\sides_k}\jump{\nabla
                                        v_k}\mean{\phi}\ds
      +\int_{\sides_k}\theta \mean{\nabla \phi}\jump{v_k}\ds
      \\
      &\quad -\int_{\mathring\sides_k}
       \big(\vec{\beta}\cdot\jump{v_k}\jump{\nabla \phi}+\jump{\nabla
         v_k}\vec{\beta}\cdot\jump{\phi}\big)\ds
       \\
    &\quad - \int_\Omega \gamma
      \big(\riftG[k](\jump{v_k})+\liftG[k](\vec{\beta}\cdot\jump{v_k})\big)\cdot
      \big(\riftG[k](\jump{\phi})+\liftG[k](\vec{\beta}\cdot\jump{\phi})\big)\dx
    \\
    &\quad-\sigma\int_{\sides_k}h_k^{-1}\jump{v_k}\jump{\phi}\ds.
  \end{align*}
  Thanks to properties of $\Pi_k$ (see Lemma~\ref{lem:Pik}), we have
  that $\jump{v_k}|_\side \equiv0$ for
  $S\in\sides_k\setminus\sides_k^+$, $\jump{v_k}|_{\Omega\setminus\Omega_k^+} \equiv0$,
  $\phi|_{\elm}\equiv0$ for $\elm\in\gridk^{++}$, and $\jump{\phi}|_{S}\equiv0$ for $S\in\sides_k^{++}$. Therefore, we have 
  \begin{align}\label{eq:17}
    \begin{aligned}
      \scp{f}{\phi}&-\bilin[k]{v_k}{\phi}
      \\
      &=\int_{\gridk\setminus\gridk^{++}}(f+\Delta v_k)\phi\dx
      -\int_{\sides_k\setminus\sides_k^{++}}\jump{\nabla
        v_k}\mean{\phi}\ds
      \\
      &\quad+\theta\int_{\sides_k^+}\mean{\nabla \phi}\jump{v_k}\ds
      \\
      &\quad -\int_{\mathring\sides_k^+}
        \vec{\beta}\cdot\jump{v_k}\jump{\nabla
          \phi}\ds-\int_{\mathring\sides_k\setminus\sides_k^{++}}
        \jump{\nabla v_k}\vec{\beta}\cdot\jump{\phi}\ds
      \\
      &\quad - \int_\Omega \gamma
        \big(\riftG[k](\jump{v_k})+\liftG[k](\vec{\beta}\cdot\jump{v_k})\big)\cdot
        \big(\riftG[k](\jump{\phi})+\liftG[k](\vec{\beta}\cdot\jump{\phi})\big)\dx
      \
      \\
      &\quad-\sigma\int_{\sides_k^+\setminus\sides_k^{++}}h_k^{-1}\jump{v_k}\jump{\phi}\ds
    \end{aligned}.
\end{align}    
The last term on the right-hand side of \eqref{eq:17} can be estimated using Cauchy-Schwarz' inequality;
  for the first two terms we use the interpolation estimates from
  Corollary~\ref{cor:PikIpol} for 
    $\phi=\psi-\Pi_k\psi$ with $\psi=u_k-\Pi_ku_\infty\in\V_k$ as to
    obtain 
    \begin{multline*}
      \int_{\gridk\setminus\gridk^{++}}(f+\Delta
                                        v_k)\phi\dx
                                        -\int_{\sides_k\setminus\sides_k^{++}}\jump{\nabla
                                        v_k}\mean{\phi}\ds
                                      \\
                                      \Cleq \Bigg[\Big(\int_{\gridk\setminus\gridk^{++}}h_k^2|f+\Delta
      v_k|^2\dx\Big)^{1/2}
      + \Big(\int_{\sides_k\setminus\sides_k^{++}} h_k\jump{\nabla
      v_k}^2\ds\Big)^{1/2}\Bigg]\enorm[k]{u_k-\Pi_ku_\infty}. 
    \end{multline*}
    Moreover, from $\phi|_\elm\equiv0$, $\elm\in\gridk^{++}$, we have
    that $\phi|_{\omegak(S)}\equiv0$ and thus
      $\mean{\nabla\phi}|_\side\equiv0$ for all 
    $\side\in\sides_k^{3+}=\sides(\gridk^{3+})$. Therefore, by
    standard trace inequalities, inverse estimates and
    Corollary~\ref{cor:PikIpol}, we have  that
    \begin{align*}
      \int_{\sides_k^+}\mean{\nabla \phi}\jump{v_k}\ds&=
                                                        \int_{\sides_k^+\setminus\sides_k^{3+}}\mean{\nabla
                                                        \phi}\jump{v_k}\ds
                                                        \Cleq
                                                        \Big(\int_{\sides_k^+\setminus\sides_k^{3+}}
                                                        h_k^{-1}\jump{v_k}^2\ds\Big)^{1/2}   
                                                        \enorm[k]{\phi}.
     \end{align*}
     A similar argument yields 
     \begin{align*}
       \int_{\mathring\sides_k^+}
       \vec{\beta}\cdot\jump{v_k}\jump{\nabla
        \phi}\ds&=\int_{\mathring\sides_k^+\setminus\sides_k^{3+}}
       \vec{\beta}\cdot\jump{v_k}\jump{\nabla
        \phi}\ds 
       \\
       &\Cleq \abs{\vec{\beta}}
                                                        \Big(\int_{\mathring\sides_k^+\setminus\sides_k^{3+}}
                                                        h_k^{-1}\jump{v_k}^2\ds\Big)^{1/2}   
                                                        \enorm[k]{\phi}.
     \end{align*}
     Finally we have with~\eqref{eq:liftGstab} and the local support
     of the local liftings, that
     \begin{align*}
       \int_\Omega 
      \riftG[k](\jump{v_k})\cdot
      \riftG[k](\jump{\phi})\dx &=\int_\Omega
       \big(\sum_{\side\in \sides_k^+
       }\riftS[k](\jump{v_k})\big)
       \cdot\big(\sum_{\side\in\sides_k\setminus\sides_k^{++} 
       }\riftS[k](\jump{\phi})\big)\dx
                                  \\
       &= \int_{\gridk^+\setminus\gridk^{++}}\riftG[k](\jump{v_k})\cdot
      \riftG[k](\jump{\phi})\dx 
         \\
       &\Cleq
         \big(\int_{\sides_k^{+}\setminus\sides_k^{3+}}\hk^{-1}\jump{v_k}^2\ds
         \Big)^{1/2}  \enorm[k]{\phi}.
     \end{align*}
     Similar bounds hold for the remaining terms in~\eqref{eq:17}.
     Combining the above observations proves the desired assertion.
\end{proof}

% \subsection{The \ADGM limit}
% \label{sec:adgm-limit}
In order to conclude convergence of the sequence of discrete discontinuous
Galerkin approximations from Lemma~\ref{lem:conv2}, we need to control
the error estimator. To this end, we shall use Verf\"urth's bubble
function technique. 

Let $n\in\N$, such that 
$n$ uniform refinements of an element ensure that the element as well
as  each of its sides have at least one interior node.
We specify the elements in $\gridk$, which neighbourhood is eventually uniformly
refined $n$ times by
\begin{align*}
  \grid_k^0\definedas \big\{\elm\in\gridk\colon& \exists
  \ell=\ell(\elm)\ge k+n~\text{such that} \\
  &\text{all $\elm'\in\neighk(\elm)$ are $n$ 
times uniformly refined in $\grid_\ell$} \big\}
\end{align*}
This
guarantees that suitable discrete interior and side bubble
functions are available in $\V_\infty$ for all $\elm\in\gridk^0$;
(compare also with \cite{Doerfler:96}, \cite{MoNoSi:00} and
\cite{MorinSiebertVeeser:08}). 
We define $\Omega_k^0:=\Omega(\gridk^0)% \bigcup\{\omegak(\elm):\elm\in\gridk^0\}
\subset\Omega_k^-$. Introducing $\Omega_k^\star=\Omega(\gridk^\star)$ with
$\gridk^\star\definedas\gridk\setminus(\gridk^{++}\cup\gridk^0)$,
we have from 
\cite[(4.15)]{MorinSiebertVeeser:08} that
\begin{align}
  \label{eq:MSV(4.15)}
  |\Omega(\gridk^\star)|% =|\Omega\setminus(\Omega_k^0\cup\Omega_k^{++})|
  \to 0 \qquad\text{as}~k\to\infty.
\end{align}

\begin{prop}\label{prop:lower8}
  Let $u_\infty$ be the solution of~\eqref{eq:u8}. 
  Then, for every $\elm\in\gridk^0$ and $v\in\Vk$, $k\in\N$, we have
  \begin{multline*}
    \int_{\elm}h_k^2|f+\Delta
      v|^2\dx+\int_{\partial\elm\cap\Omega} h_k\jump{\nablaG
        v}^2\ds\\
      \begin{split}
        &\Cleq \norm[\omegak(\elm)]{\nablaG (u_\infty-v)}^2
        +\int_{\{S\in\sides^+:S\subset\omegak(\elm)\}}h_+^{-1}\jump{u_\infty-v}^2\ds
        \\
        &\qquad+\osc(\neighk(\elm),f)^2;
      \end{split}
  \end{multline*}
  in particular, we also have 
  \begin{multline*}
    \sum_{\elm\in\grid_k^0}\int_{\elm}h_k^2|f+\Delta
      v|^2\dx+\int_{\partial\elm\cap\Omega} h_k\jump{\nablaG
        v}^2\ds% \est_k(\elm,v)^2
    \\\Cleq \enorm[\infty]{u_\infty-v}^2+\sum_{\elm\in\gridk^0}\sum_{\elm'\in\omega_k(\elm)}\osc(\elm',f)^2.
  \end{multline*}
  Note that since $v\in\Vk\not\subset\V_\infty$ in general, the above terms may be equal to infinity.
\end{prop}

\begin{proof}
  The proof follows from standard techniques; compare
  e.g. \cite{KarakashianPascal:03,BonitoNochetto:10} by replacing
  Verf\"urth's bubble functions by their discrete counterparts. 
  However, in order to keep the presentation self-contained, we
  provide a sketch of the proof. Let $\elm\in \gridk^0$, then, thanks
  to the  definition of $\gridk^0$, there exists some $\ell>k$ such that there exists a
  piecewise affine discrete bubble function
  $\phi_\elm\in\V_\ell\cap C(\Omega)$ 
  satisfying $ \phi_\elm\in H_0^1(\elm),$ and
  \begin{align}\label{eq:phiE}
    h_\elm^{{d}}\norm[L^\infty(\elm)]{\nabla q\phi_\elm}^2\Cleq\norm[\elm]{\nabla q\phi_\elm}^2\Cleq
    h_\elm^{-2}\norm[\elm]{q}^2\quad\text{for all}~q\in\P_{r-1} (E);
  \end{align}
  compare \cite{Doerfler:96,MorinSiebertVeeser:08}.
  Let $f_\elm\in\P_{r-1}(E)$ an arbitrary polynomial. Observing that
  $(f_\elm+\Delta v)\phi_\elm\in C(\Omega)$ and thus does not jump across
  sides, we have by
  equivalence of norms on finite dimensional spaces and a scaled trace inequality, that 
  \begin{multline*}
    \int_\elm |f_\elm+\Delta v|^2\dx
    \\
    \begin{aligned}
      &\Cleq \int_\elm (f_\elm+\Delta v) (f_\elm+\Delta v)\phi_\elm\dx
      \\
      &= \bilin[\infty]{u_\infty-v}{(f_\elm+\Delta v)\phi_\elm}-\int_\elm
      (f-f_\elm) (f_\elm+\Delta v)\phi_\elm\dx
      \\
      &\Cleq 
      \norm[\elm]{\nablaG (u_\infty - v)}
      \norm[\elm]{\nabla(f_\elm+\Delta v)\phi_\elm}
      - \int_{\sides^+}\jump{u_\infty-v}\mean{\nabla(f_\elm+\Delta v)\phi_\elm}\ds\\
      &\quad+\norm[\elm]{f-f_\elm}\norm[\elm]{(f_\elm+\Delta v)\phi_\elm}.
    \end{aligned}
  \end{multline*}
  From~\eqref{eq:phiE} and standard inverse estimates,  we conclude that 
  \begin{multline*}
    \abs{\int_{\sides^+}\jump{u_\infty-v}\mean{\nabla(f_\elm+\Delta
    v)\phi_\elm}\ds}\\
    \begin{aligned}
      &\le \sum_{S\in \sides^+,S\subset
        \elm}\int_{S}\jump{u_\infty-v}^2\ds\norm[L^\infty(\elm)]{\nabla(f_\elm+\Delta
    v)\phi_\elm}
        \\
      &\Cleq
      \Big(\int_{\sides^+}h_+^{{d-1}}\jump{u_\infty-v}^2\ds\Big)^{1/2}
      h_\elm^{-1-\frac{d}2}\norm[\elm]{f_\elm+\Delta
    v} 
  \\
  &\Cleq
  \Big(\int_{\sides^+}h_+^{{-1}}\jump{u_\infty-v}^2\ds\Big)^{1/2}
  h_\elm^{-1}\norm[\elm]{f_\elm+\Delta
    v}, 
    \end{aligned}
  \end{multline*}
  since ${h_+}\le h_\elm$ on $\elm$.
  Therefore, we arrive at 
  \begin{align}\label{eq:est_elm}
    \begin{aligned}
      \int_\elm h_k^2|f_\elm+\Delta v|^2\dx &\Cleq \norm[\elm]{\nablaG
        (u_\infty - v)}^2+ \sum_{S\in \sides^+,S\subset
        \elm}\int_\side h_+^{-1}\jump{u_\infty-v}^2\ds
      \\
      &\quad+h_\elm^2\norm[\elm]{f-f_\elm}^2.
    \end{aligned}
  \end{align}
  Thanks to the definition of $\gridk^0$, the same bound applies for
  all $\elm'\in\neighk(\elm)$.

  We now turn to investigate the jump terms. To this end, we fix
  one $\side\in\mathring\sides_k$, $\side\subset\elm$. Thanks to the definition of
  $\gridk^0$, there exists % some interior node $z\in\nodes_\ell$ of
  % $\side$ and
  a hat function $\phi_\side \in
  \V_\ell\cap C(\Omega)\cap H_0^1(\omegak(S))$,
  and for 
  $q\in\P_{r-1}(\side)$, let $\tilde
  q\in\P_{r-1}(\omegak(S))$ be some extension, such that
  \begin{align}\label{eq:phiS}
   h_\elm^d \norm[L^\infty(\omegak(S))]{\nabla\tilde q\phi_\side}\Cleq\norm[\omegak(S)]{\tilde q\phi_\side}^2\Cleq h_\elm \int_\side  |q|^2\ds.
  \end{align}
  Noting that $\jump{\nabla v}\in\P_{r-1}(S)$, we
  have, by the equivalence of norms on finite dimensional spaces, that 
  \begin{align*}
    \int_{S} \jump{\nabla
        v}^2\ds
      &\Cleq \int_{S} \jump{\nabla v}^2\phi_\side\ds
      \\
      &= 
      \bilin[\infty]{u_\infty-v}{\widetilde{\jump{\nabla
            v}}\phi_\side} - \int_{\omega_k(S)}
      (f+\Delta v)\widetilde{\jump{\nabla v}}\phi_\side\dx
      \\
      &\Cleq
      \norm[\omegak(S)]{\nablaG(u_\infty-v)}\norm[\omegak(S)]{\nabla\widetilde{\jump{\nabla
            v}}\phi_\side}
      \\
      &\quad + \int_{\sides^+} \jump{u_\infty-v}\mean{\nabla\widetilde{\jump{\nabla
            v}}\phi_\side}\ds
        \\
    &\quad+\big(\norm[\elm]{f+\Delta v}^2+\norm[\elm']{f+\Delta v}^2\big)^{\frac12}\norm[\omegak(S)]{\widetilde{\jump{\nabla
            v}}\phi_\side}.
  \end{align*}
   Similarly, as for the element residual, we have that 
  \begin{multline*}
    \int_{\sides^+} \jump{u_\infty-v}\mean{\nabla\widetilde{\jump{\nabla
            v}}\phi_\side}\ds
      \\
      \Cleq\Big(\sum_{\side'\in\sides^+,
    \side'\subset\omegak(S)}h_+^{-1}\jump{u_\infty-v}^2\Big)^{\frac12}
    \Big(\int_\side h_\elm\jump{\nabla v}^2\ds\Big)^{\frac12},
  \end{multline*}
using \eqref{eq:phiS}.  Combining this with \eqref{eq:phiS}, we obtain
  \begin{align*}
    \int_{S}h_\elm \jump{\nabla
        v}^2\ds
    &\Cleq\norm[\omegak(S)]{\nablaG(u_\infty-v)}^2+ \sum_{\side'\in\sides^+,
      \side'\subset\omegak(S)}\int_\side h_+^{-1}\jump{u_\infty-v}^2 \ds
      \\
    &\quad + h_\elm^2\norm[\elm]{f+\Delta v}^2+h_{\elm'}^2\norm[\elm']{f+\Delta
      v}^2. 
  \end{align*}
  Finally applying the bound \eqref{eq:est_elm} to
  $\elm,\elm'\in\neighk(\elm)$, we have proved the first assertion. 

  The second assertion follows, then, by summing over all
  $\elm\in\gridk^0$ together with an observation from
  \cite{MorinSiebertVeeser:08}, which we sketch here in order to keep the
  this work self-contained. Let
  $M:=\max\{\#\neighk(\elm):\elm\in\gridk^0\}$ 
  be the maximal number of neighbours, then $\gridk^0$ can be split
  into $M^2+1$ subsets $\grid_{k,0}^0,\ldots, \grid_{k,M^2}^0$ such
  that for each $j$, we have that $\elm',\elm\in \grid_{k,j}^0$ with
  $\elm\neq\elm'$ implies that $\neighk(\elm)\cap\neighk(\elm')=\emptyset$.
  Consequently, we have 
  \begin{align*}
    \sum_{\elm\in\gridk^0} \norm[\omegak(\elm)]{\nablaG
      (u_\infty-v)}^2 
    &\le  \sum_{j=0}^{M^2}\sum_{\elm\in\grid_{k,j}^0} \norm[\omegak(\elm)]{\nablaG
      (u_\infty-v)}^2
    \\
    &\le (M^2+1) \norm[\Omega_k^0]{\nablaG
      (u_\infty-v)}^2.
  \end{align*}  
  Together with similar estimates for the jump terms and the
  oscillations  the second assertion follows from the first one.
\end{proof}

\begin{thm}\label{Thm:uk->u8}
  Let $u_\infty$ the solution of~\eqref{eq:u8} and $u_k\in\Vk$ be the
\DGFEM approximation from \eqref{ipdg} on $\gridk$ for some $k\in\N$.
  Then, 
  \begin{align*}
    \enorm[k]{u_\infty-u_k}\to 0 \quad\text{as}~k\to\infty
  \end{align*}
  and in particular \(\norm[\Omega]{u_\infty-u_k}\to 0\) as \(k\to\infty\).
\end{thm}

\begin{proof}
  Thanks to Lemma~\ref{lem:conv1}, Lemma~\ref{lem:Pik2} and
  Lemma~\ref{lem:conv2}, we have that 
  \begin{align*}
    \lim_{k\to\infty}\enorm[k]{u_\infty-u_k}^2&\Cleq
    \lim_{k\to\infty}\enorm[\infty]{u_\infty-v_k}^2+\sum_{\elm\in\grid_k\setminus\grid_k^{3+}}\est_k(v_k,\elm)^2
    \\
    &=\lim_{k\to\infty}\sum_{\elm\in\grid_k\setminus\grid_k^{3+}}\est_k(v_k,\elm)^2,
  \end{align*}
  where $v_k:=\Pi_ku_\infty$.
  % It is
  % proved in \cite[(4.15)]{MorinSiebertVeeser:08}???? that 
  % \begin{align*}
  %   |\Omega\setminus(\Omega_k^0\cup\Omega_k^{++})|\to 0\qquad\text{as}~k\to\infty.
  % \end{align*}
  % Using Lemma~\ref{lem:Omegastar} w
  We conclude
  from~\eqref{eq:MSV(4.15)} that 
  \begin{align*}
   \big|\Omega\setminus\big(\Omega_k^0\cup\Omega_k^{3+}\big)\big|&\le
   \big|\Omega\setminus(\Omega_k^0\cup\Omega_k^{++})\big|+|\Omega_k^{++}\setminus\Omega^{3+}_k| 
    % \\
    % &\le
    %   \big|\Omega_k^\star\big|+|\Omega^{+}\setminus\Omega^{3+}_k|
      \to 0 ,
  \end{align*}
  as $k\to\infty$. Indeed, for $k\in\N$, it
  follows from Lemma~\ref{L:Nk=NK} and 
  $\#\gridk[k]^{+}<\infty$, that there exists $K=K(k)$, such that 
  $\gridk[k]^{+}\subset\grid_K^{3+}$,
  i.e. $|\Omega^+\setminus\Omega^{3+}_K|\le
  |\Omega^+\setminus\Omega^{+}_k|\to0$ as $k\to\infty$.
  Thanks to monotonicity we conclude
  that $|\Omega^{++}_k\setminus\Omega^{3+}_k|\le|\Omega^{+}\setminus\Omega^{3+}_k| \to 0$ as $k\to\infty$.
  We next show that this implies
  \begin{align*}
    \sum_{\elm\in\grid_k\setminus(\grid_k^0\cup\grid_k^{3+})}
    \est_k(v_k,\elm)^2\to
    0. 
  \end{align*}
  Lemma~\ref{lem:Pik2} implies that
  $\enorm[\infty]{u_\infty-v_k}\to0$ and, thus, the interior
  residual and the gradient
  jumps part of the estimator vanish due to uniform
  integrability. Moreover, it follows from Proposition~\ref{prop:V} that
  \begin{align*}
    \int_{\sides(\gridk\setminus(\grid_k^0\cup\grid_k^{3+}))}
    h_k^{-1}\jump{v_k}^2\ds
    &\Cleq  
    \int_{\sides(\gridk\setminus\gridk^{3+})}
    h_k^{-1}\jump{u_\infty}^2\ds+\enorm[k]{u_\infty-v_k}^2
      \\
    &\le \int_{\sides(\grid^+\setminus\gridk^{3+})}
    h_+^{-1}\jump{u_\infty}^2\ds+\enorm[k]{u_\infty-v_k}^2.
    % \\
    % &\to0
  \end{align*}
  % as $k\to\infty$. 
  The last term on the right-hand side of the above estimate vanishes thanks to Lemma~\ref{lem:Pik2}. 
  Again, letting $K=K(k)$, such that 
  $\gridk^{+}\subset\grid_K^{3+}$, we have
  \begin{align*}
    \int_{\sides(\grid^+\setminus\gridk[K(k)]^{3+})}
    h_+^{-1}\jump{u_\infty}^2\ds\le \int_{\sides(\grid^+\setminus\gridk^{+})}
    h_+^{-1}\jump{u_\infty}^2\ds\to 0,\quad\text{as}~k\to\infty. 
  \end{align*}
  Thanks to monotonicity, we thus conclude $\int_{\sides(\grid^+\setminus\gridk^{3+})}
    h_+^{-1}\jump{u_\infty}^2\ds\to0$, as $k\to\infty$.

  On the remaining elements $\gridk^-$, it follows from Proposition~\ref{prop:lower8} 
  that 
  \begin{align*}
    \sum_{\elm\in\grid_k^0}\est_k(v_k,\elm)^2\Cleq
    \enorm[\infty]{u_\infty-v_k}^2+
    \sum_{\elm\in
\grid_k^0}\osc(\neighk(\elm),f)^2.
  \end{align*}
  The first term on the right-hand side vanishes due to Lemma~\ref{lem:Pik2}. For the
  second term we observe that $|\bigcup\{\omegak(\elm):\elm\in
  \gridk^0\}|\Cleq| \Omega_k^0|$, depending on the shape regularity of
  $\grid_0$ and, therefore, it vanishes since 
  \begin{align}\label{eq:oscG0}
    \norm[L^\infty(\Omega)]{h_k\chi_{\Omega_k^0}} \le
    \norm[L^\infty(\Omega)]{h_k\chi_{\Omega_k^-}} 
    \to 0 \quad\text{as}~k\to\infty,
  \end{align}
  thanks to Lemma~\ref{lem:Omegastar}.

  The second limit follows then from
    \begin{align*}
      \norm[\Omega]{u_\infty-u_k}&\le \norm[\Omega]{u_\infty-\Pi_k
                                   u_\infty}+\norm[\Omega]{\Pi_ku_\infty-u_k}
      \\
                                 &\Cleq \enorm[\infty]{u_\infty-\Pi_k
      u_\infty}+\enorm[k]{\Pi_ku_\infty-u_k},
    \end{align*}
  which vanishes due to the above observations.
  \end{proof}

\section{Proof of the main result}
\label{sec:convergence-est}

We are now in the position to prove that the error estimator
vanishes by splitting the estimator according to 
\begin{align}\label{eq:splitting}
  \gridk=\gridk^0\cup\gridk^{++}\cup\gridk^\star
\end{align}
and consider each part separately following the ideas of \cite{MorinSiebertVeeser:08}.
This in turn implies that the sequence of discontinuous Galerkin
  approximations produced by \ADGM indeed converges to the exact
  solution of \eqref{eq:elliptic}.

\begin{lem}\label{lem:Gk0}
  We have that 
  \begin{align*}
    \est_k(\gridk^0)\to 0 ,\quad\text{as}~k\to\infty.
  \end{align*}
\end{lem}
\begin{proof}
  Thanks to Proposition~\ref{prop:lower8}, we have 
  \begin{multline*}
    \sum_{\elm\in\grid_k^0}\int_{\elm}h_k^2|f+\Delta
      \uk|^2\dx+\int_{\partial\elm\cap\Omega} h_k\jump{\nabla
        \uk}^2\ds% \est_k(\elm,v)^2
    \\
    \Cleq \enorm[\infty]{u_\infty-\uk}^2+\sum_{\elm\in\gridk^0}\osc(\neighk(\elm),f)^2.
  \end{multline*}
  The right-hand side vanishes thanks to Theorem~\ref{Thm:uk->u8}
  and~\eqref{eq:oscG0}. 
  % and
  % the fact that
  % \begin{align*}
  %   \lim_{k\to\infty}\norm[L^\infty(\Omega)]{h_k\chi_{\Omega_k^0}}=0;
  % \end{align*}
  % see \cite[Corollary~4.1]{MorinSiebertVeeser:08}. 

  It remains to prove that 
  \begin{align*}
    \int_{\sides(\gridk^0)}h_k^{-1}\jump{\uk}^2\ds\to0,\quad\text{as}~k\to\infty.
  \end{align*}
  By definition,
  $\Omega_k^0\subset\Omega\setminus\Omega_k^+$ and, thanks to
  Lemma~\ref{lem:Pik}\eqref{Pik:3}, we have that $\Pi_k u_\infty\in C(\Omega\setminus\Omega_k^+)$. 
  Therefore, we conclude
  \begin{align*}
    \int_{\sides(\gridk^0)}h_k^{-1}\jump{\uk}^2\ds&=
    \int_{\sides(\gridk^0)}h_k^{-1}\jump{\uk-\Pi_ku_\infty}^2\ds
                                                    \le \enorm[k]{\uk-\Pi_ku_\infty}\to0
  \end{align*}
  as $k\to\infty$; see Lemma~\ref{lem:Pik2} and Theorem~\ref{Thm:uk->u8}.
\end{proof}

\begin{lem}\label{lem:Gk*}
  We have that 
  \begin{align*}
    \lim_{k\to\infty}\est_k(\gridk^\star)=0.
  \end{align*}
\end{lem}
\begin{proof}
  We conclude from the lower bound (Proposition~\ref{prop:lower}) that 
  \begin{multline*}
    \sum_{\elm\in\gridk^\star}\int_{\elm}h_k^2|f+\Delta
      \uk|^2\dx+\int_{\partial\elm} h_k\jump{\nabla
        \uk}^2\ds
      \\
      \begin{split}
        &\Cleq
        \sum_{\elm\in\gridk^\star}\norm[\omega_k(\elm)]{u-\uk}^2+\norm[\omega_k(\elm)]{\nabla
          u-\nablaG\uk}^2+\osc(\neighk(\elm),f)^2
        \\
        % &\quad+\int_{\sides(\gridk^\star)}h_k^{-1}\jump{\uk}^2\ds.
        % \\
        &\Cleq
        \sum_{\elm\in\gridk^\star}\Big\{\norm[\omega_k(\elm)]{u}^2+\norm[\omega_k(\elm)]{u_\infty-\uk}^2+\norm[\omega_k(\elm)]{u_\infty}^2
        \\
        &\qquad+\norm[\omega_k(\elm)]{\nabla
          u}^2+\norm[\omega_k(\elm)]{\nablaG
          u_\infty-\nablaG\uk}^2+\norm[\omega_k(\elm)]{\nablaG
          u_\infty}^2 \\
        &\qquad+\osc(\neighk(\elm),f)^2\Big\}.
        % \\
        % &\quad+\int_{\sides(\gridk^\star)}h_k^{-1}\jump{\uk}^2\ds.
      \end{split}
  \end{multline*}
  This vanishes as $k\to\infty$ thanks to Theorem~\ref{Thm:uk->u8} and~\eqref{eq:MSV(4.15)},
  together with the uniform integrability of the terms involving $u$
  and $u_\infty$. Note that
  $\big|\bigcup\{\omegak(\elm):\elm\in\gridk^\star\}\big|\Cleq
  |\Omega_k^\star|$, with the constant depending on the shape regularity of $\grid_0$.

  It remains to prove
  \begin{align*}
    \int_{\sides(\gridk^\star)}h_k^{-1}\jump{\uk}^2\ds\to0,\quad\text{as}~k\to\infty.
  \end{align*}
  To this end, we observe that
  \begin{align*}
   \int_{\sides(\gridk^\star)}h_k^{-1}\jump{\uk}^2\ds&=
    \int_{\sides(\gridk^\star)}h_k^{-1}\jump{\uk-\Pi_ku_\infty}^2\ds+\int_{\sides(\gridk^\star)}h_k^{-1}\jump{\Pi_ku_\infty}^2\ds
    \\
    &\le \frac1{\bar\sigma}\enorm[k]{\uk-\Pi_ku_\infty}^2+\int_{\sides(\gridk^\star)}h_k^{-1}\jump{\Pi_ku_\infty}^2\ds.
  \end{align*}
  As in the proof of Lemma~\ref{lem:Gk0}, we have that the first term
  vanishes as $k\to\infty$. Thanks to Lemma~\ref{L:Nk=NK}, there
  exists $\ell(k)\ge K(k)\ge k$ such that
  $\gridk^+\subset\gridk[K(k)]^{++}$ and
  $\gridk[K(k)]^+\subset\gridk[\ell(k)]^{++}$.
  Consequently, we have that $\jump{\Pi_\ell u_\infty}|_\side =0$ for all
  $S\in \gridk$; see Lemma~\ref{lem:Pik}\eqref{Pik:3}.  
  Therefore, we conclude from Lemma~\ref{lem:Pik2} that
  \begin{align*}
    \sigma\int_{\sides(\gridk^\star)}h_k^{-1}\jump{\Pi_ku_\infty}^2\ds&=\sigma\int_{\sides(\gridk^\star)}h_k^{-1}\jump{\Pi_ku_\infty-\Pi_\ell
    u_\infty}^2\ds
    \\
    &\Cleq \enorm[k]{\Pi_ku_\infty-
    u_\infty}^2+\enorm[\ell]{u_\infty-\Pi_\ell
    u_\infty}^2\to0,
  \end{align*}
  as $k\to\infty$.
\end{proof}

\begin{lem}\label{lem:Gk+}
  We have 
  \begin{align*}
    \est_k(\gridk^{++})\to0\quad\text{as}~k\to\infty.
  \end{align*}
\end{lem}
\begin{proof}\textbf{Step 1:} By definition, elements in $\gridk^{++}$ will not be
  subdivided, i.e. we have that
  $\mathcal{M}_k\subset\gridk\setminus\gridk^{++}$; compare
  with~\eqref{eq:refine}. As a consequence of Lemmas~\ref{lem:Gk0}
  and~\ref{lem:Gk*},
  we conclude from~\eqref{eq:mark} for all
  $\elm\in\gridk^{++}$
  that 
  \begin{align}\label{eq:E->0}
    \est_k(\elm)\le \lim_{k\to\infty}g(\est_k(\mathcal{M}_k))=
    \lim_{k\to\infty} g(\est_k(\gridk^-\cup\grid_k^\star))\to0,
  \end{align}
  as $k\to\infty$. We shall reformulate the above element-wise convergence
  in an integral framework, in order to conclude $\est_k(\gridk^{++})\to 0$
  as $k\to\infty$ via a generalised version of the dominated
  convergence theorem. To this end, we shall consider some
  properties of the error indicators.
  
  \textbf{Step 2:}
  Thanks to the definition of $\gridk^{++}$, we have for all
  $\elm\in\gridk^{++}$, that
  $\omega_k(\elm)=\omega_\ell(\elm)=:\omega(\elm)$ and
  $\neighk(\elm)=\neighk[\ell](\elm)=\neigh(\elm)$ for all $\ell\ge k$.
  Therefore, we obtain by the lower bound,
  Proposition~\ref{prop:lower}, that 
  \begin{align}\label{eq:estBND}
    \begin{aligned}
      \est_k(\elm)^2&\Cleq
      \enorm[\neigh(\elm)]{\uk-u}^2+\osc(\neigh(\elm),f)^2
      \\
      &\Cleq
      \enorm[\neigh(\elm)]{\uk-u_\infty}^2+\norm[\neigh(\elm)]{u_\infty}^2+\norm[H^1(\omega(\elm))]{u}^2
      +\norm[\omega(\elm)]{f}^2
      \\
      &=:\enorm[\neigh(\elm)]{\uk-u_\infty}^2+C_\elm^2.
    \end{aligned}
  \end{align}
  Arguing as in the proof of Proposition~\ref{prop:lower8}, 
  we can conclude from the local estimate that
  \begin{align}\label{eq:CEBND}
    \sum_{\elm\in\gridk^{++}}C_\elm^2 \Cleq \enorm[\infty]{u_\infty}^2+\norm[H^1(\Omega)]{u}^2+\|f\|_{\Omega}^2<\infty
  \end{align}
  independently of $k$.
  
  \textbf{Step 3:} We shall now reformulate $\est_k(\gridk^{++})$ in
  integral form. Note that thanks to Lemma~\ref{L:Nk=NK}, we have that
  $\grid^+=\bigcup_{k\in\N_0}\gridk^+=\bigcup_{k\in\N_0}\gridk^{++}$,
  and also that the sequence $\{\gridk^{++}\}_{k\in\N_0}$ is nested. For
  $x\in\Omega^+$, let 
  \begin{align*}
    \ell=\ell(x):=\min\{k\in\N_0: ~\text{there
    exists}~\elm\in\grid_k^{++}~\text{such that}~x\in\elm\}.
  \end{align*}
  Then, we define 
  \begin{gather*}
    \epsilon_k(x):=M_k(x):=0\quad\text{for}~k<\ell,
    \intertext{and}
    \epsilon_k(x):=\frac1{|\elm|}\est_k^2(\elm),\qquad M_k:=\frac1{|\elm|}\Big(\enorm[\neigh(\elm)]{\uk-u_\infty}^2+C_\elm^2\Big)\quad\text{for}~k\ge\ell.
  \end{gather*}
  Consequently, for any $k\in\N_0$, we have 
  \begin{align*}
    \est_k(\gridk^{++})^2=\int_{\Omega^+}\epsilon_k(x)\dx.
  \end{align*}
  Moreover, thanks to the fact that the sequence
  $\{\gridk^{++}\}_{k\in\N_0}$ is nested, we conclude from~\eqref{eq:E->0} that 
  \begin{align*}
    \lim_{k\to\infty}\epsilon_k(x)=\lim_{k\to\infty} \frac1{|\elm|}\est_k^2(\elm)=0.
  \end{align*}
  It follows from \eqref{eq:estBND} and \eqref{eq:CEBND}
  that $M_k$ is an integrable majorant for $\epsilon_k$.

\textbf{Step 4:} We shall show that the majorants $\{M_k\}_{k\in\N_0}$
converge in $L^1(\Omega^+)$ to 
\begin{align*}
  M(x):=\frac1{|\elm|}C_\elm^2, \quad\text{for}~x\in
  \elm\quad\text{and}\quad\elm\in\grid^+. 
\end{align*}
Then the assertion follows from a generalised majorised convergence
theorem; see \cite[Appendix (19a)]{Zeidler:90}. 
In fact, by the definition of $M_k$, we have that 
\begin{align*}
  \norm[L^1(\Omega^+)]{M_k-M}=\sum_{\elm\in\gridk^{++}}\norm[L^1(\elm)]{M_k-M}+\sum_{\elm\in\grid^+\setminus\gridk^{++}}\norm[L^1(\elm)]{M}. 
\end{align*}
The latter term vanishes since it is the tail of a converging series (compare
with~\eqref{eq:CEBND}) and for the former term, we have, thanks to
Theorem~\ref{Thm:uk->u8}, that
\begin{align*}
  \sum_{\elm\in\gridk^{++}}\norm[L^1(\elm)]{M_k-M}=\sum_{\elm\in\gridk^{++}}\enorm[\neigh(\elm)]{\uk-u_\infty}^2\Cleq
  \enorm[k]{u_k-u_\infty}\to 0
\end{align*}
as $k\to\infty$.
\end{proof}

\begin{proof}[Proof of Theorem~\ref{thm:main}]
  The assertion follows from Proposition~\ref{prop:upper} together
  with Lemmas~\ref{lem:Gk0}, \ref{lem:Gk*}, and
  \ref{lem:Gk+} recalling
  the splitting~\eqref{eq:splitting}.
\end{proof}

\section*{Acknowledgement}
We thank the anonymous referee of  the paper
\cite{DominicusGaspozKreuzer2019} for finding a highly non-trivial
counterexample to the first statement in
\cite[Lemma~11]{DominicusGaspozKreuzer2019}, which lead to this
corrected version of \cite{KreuzerGeorgoulis:17}.

\iftrue
\providecommand{\bysame}{\leavevmode\hbox to3em{\hrulefill}\thinspace}
\providecommand{\MR}{\relax\ifhmode\unskip\space\fi MR }
\providecommand{\MRhref}[2]{%
  \href{http://www.ams.org/mathscinet-getitem?mr=#1}{#2}
}
\providecommand{\href}[2]{#2}

% \bibliographystyle{amsalpha} 
% \bibliography{ADG-convergence}

\end{document}

\else

% bibliography ----------------------------------------------------
\bibliographystyle{amsalpha} 
\bibliography{ADG-convergence}
\fi

\end{document}